\documentclass{amsart}
  \usepackage{amsmath}
  \usepackage{paralist}
  \usepackage{graphics}
  \usepackage{epsfig}
\usepackage{amsthm,amssymb}

\begin{document}

\newtheorem{theorem}{Theorem}[section]
\newtheorem{corollary}{Corollary}
\newtheorem*{main}{Main Theorem}
\newtheorem{lemma}[theorem]{Lemma}
\newtheorem{proposition}{Proposition}
\newtheorem{conjecture}{Conjecture}
\newtheorem*{problem}{Problem}
\theoremstyle{definition}
\newtheorem{definition}[theorem]{Definition}
\newtheorem{remark}{Remark}
\newtheorem*{notation}{Notation}
\newcommand{\ep}{\varepsilon}
\newcommand{\eps}[1]{{#1}_{\varepsilon}}

\newcommand{\Real}{\mathbb{R}}
\newcommand{\e}{\varepsilon}
\newcommand{\Om}{\Omega}
\newcommand{\Ome}{\Omega^\e}
\newcommand{\om}{\omega}
\newcommand{\lm}{\lambda}
\newcommand{\lmd}{\lambda^D}
\newcommand{\f}{\varphi}
\newcommand{\n}{\nonumber}

\newcommand{\ue}{u_\e}
\newcommand{\ve}{v_\e}
\newcommand{\we}{w_\e}
\newcommand{\reps}{\rho_\e}
\newcommand{\aeps}{a_\e}
\newcommand{\tQ}{\widetilde{Q}}
\newcommand{\tG}{{\Gamma}}
\newcommand{\tGe}{{\Gamma}^\e}
\newcommand{\Ue}{U_\e}
\newcommand{\Ve}{V_\e}
\newcommand{\We}{W_\e}

\newcommand{\lme}{\lambda^\varepsilon}
\newcommand{\mue}{\mu_\varepsilon}
\newcommand{\ome}{{\om^\varepsilon}}
\renewcommand{\varrho}{\rho}

\newcommand{\intl}{\int}
\newcommand{\D}{\displaystyle}
\newcommand{\eto}{\e\to 0}
\newcommand{\di}{\, \mbox{div}\,}
\newcommand{\dist}{\, \mbox{dist}\,}

\newcommand{\nx}{n_x}
\newcommand{\ny}{n_y}
\newcommand{\yin}{\Big|_{y\in\Gamma}}
\newcommand{\Ah}{A^{\mbox{\scriptsize hom}}}
\newcommand{\p}{\partial}

\newcommand{\cL}{{\mathcal{L}}}
\newcommand{\cH}{{\mathcal{H}}}
\newcommand{\cA}{{\mathcal{A}}}
\newcommand{\cC}{{\mathcal{C}}}
\newcommand{\cD}{{\mathcal{D}}}
\newcommand{\cR}{{\mathcal{R}}}
\newcommand{\cI}{{\mathcal{I}}}
\newcommand{\cQ}{{{Q}}}
\newcommand{\cV}{{\mathcal{V}}}
\newcommand{\cN}{{\mathcal{N}}}
\newcommand{\cM}{{\mathcal{M}}}
\newcommand{\cP}{{\mathcal{P}}}
\newcommand{\cK}{{\mathcal{K}}}
\newcommand{\cB}{{\mathcal{B}}}

\newcommand{\Ae}{{\mathcal{A}^\e}}
\newcommand{\Be}{{\mathcal{B}_\e}}
\newcommand{\He}{{\mathcal{H}^\e}}
\newcommand{\Le}{{\mathcal{L}^\e}}

\newcommand{\be}{\begin{equation}}
\newcommand{\ee}{\end{equation}}
\newcommand{\ben}{\begin{equation*}}
\newcommand{\een}{\end{equation*}}

\newcommand{\bea}{\begin{eqnarray}}
\newcommand{\eea}{\end{eqnarray}}
\newcommand{\bean}{\begin{eqnarray*}}
\newcommand{\eean}{\end{eqnarray*}}

\newcommand{\ba}{\begin{array}}
\newcommand{\ea}{\end{array}}

\newcommand{\bpr}{\begin{proof}}
\newcommand{\epr}{\end{proof}}
\newcommand{\bl}{\begin{lemma}}
\newcommand{\el}{\end{lemma}}

\title[Homogenization with doubly high contrasts]%
{Homogenization of spectral problems\\ in bounded domains with doubly high contrasts}

\author[N.O. Babych, I.V. Kamotski and V.P. Smyshlyaev]%
{\scshape Natalia O. Babych, Ilia V. Kamotski
and Valery P. Smyshlyaev}

\medskip

\begin{abstract}
Homogenization of a spectral problem in a bounded domain with a high contrast in both
stiffness and density is considered. For a special critical scaling, two-scale asymptotic expansions for eigenvalues and
eigenfunctions are constructed. Two-scale limit equations are derived and relate to certain non-standard self-adjoint
operators. In particular they explicitly display the first two terms in the asymptotic
expansion for the eigenvalues, with a surprising bound for the error of order $\e^{5/4}$ proved.
\end{abstract}

\maketitle

{\scriptsize
\emph{2000 Mathematics Subject Classification.}
Primary: 35B27; Secondary: 34E

\emph{Key words.} Homogenization, periodic media, high-contrasts, eigenvalue asymptotics
}


 \section{Introduction}
Homogenization for problems with physical properties which are not only highly oscillatory but also highly
heterogeneous has long been documented to display
unusual effects, for example the memory effects observed by E. Ya. Khruslov \cite{FeKh, Kh, Kh2}.
Of particular interest in this context are the double-porosity models where the parameter of
high-contrast $\delta$ is critically scaled again the periodicity size $\e$, $\delta\sim\e^2$,
e.g. \cite{ADH, BMP}. Those have been treated both by a high-contrast version of the classical method of asymptotic
expansions, e.g. \cite{Panas, Sandr, Cheredn2, KS} and using the techniques of two-scale convergence,
e.g. \cite{Zhikov2000, Zh2, Cherd}.
In particular, for spectral problems in bounded \cite{Zhikov2000}
and unbounded \cite{Zh2} periodic domains
V.V.~Zhikov studied the spectral convergence,
introduced two-scale limit operator,
developed the techniques of two-scale resolvent
convergence and two-scale compactness.
In \cite{KS} the spectral convergence of
eigenvalues in the gaps of Floquet-Bloch spectrum due to defects in double-porosity type
media were studied,
and \cite{Cherd} supplemented this by the analysis of eigenfunction convergence based on an analysis of a
uniform exponential decay.

In this work we study spectral problems
of double-porosity type in a bounded domain $\Omega$ where the
high contrast might occur not only in the ``stiffness'' coefficient but also in the ``density'',
and argue that this leads to some interesting new effects.
Namely, referring to the next section for
precise technical formulations, for the spectral problem
\be
-\di(\aeps\left({x}\right) \nabla\ue)
\,=\,
\lme \reps\left({x}\right) \ue,
\label{intrprob}
\ee
with Dirichlet boundary conditions on the exterior boundary, most generally, both $a_\e$ and $\rho_\e$
are $\e$-periodic, $a_\e=\rho_\e=1$ in the connected matrix and $a_\e\sim\e^\alpha$, $\rho_\e\sim\e^\beta$
in the disconnected inclusions.
(Outside homogenization, the above resembles problems of vibrations
with high contrasts in both density and stiffness, e.g. \cite{NatGol}.)
The double-porosity corresponds to $\alpha=2$ and $\beta=0$. For
$\beta\neq 0$, it is not hard to see that it is $\alpha=\beta+2$ when the spectral problems at the
macro and micro-scales are coupled in a non-trivial way. To explore this, we choose $\beta=-1$ and $\alpha=1$
and show that this leads to some unusually coupled two-scale limit behaviors of the eigenfunctions and the
eigenvalues.

Namely, although the limit behavior of the eigenfunctions is still
somewhat similar to that of double porosity, i.e. the two-scale
limit is a function of only slow variable $x$ in the matrix and a
function of both $x$ and the fast variable $y$ in the inclusions,
the limit equations themselves are quite different. We show that
there exist asymptotic series of eigenvalues
$\lambda^\e\sim\lambda_0+\e\lambda_1$ with $\lambda_0$ being any
eigenvalue of a non-standard self-adjoint ``microscopic'' inclusion
problem, Theorem \ref{limoper1}, whose eigenfunctions are directly
related to the two-scale limit $w_0(x,y)$ in the matrix. In fact,
$\lambda_0$ is either a solution of
$\beta(\lambda_0)=|Q_1|\lambda_0$, where $\beta(\lambda)$ is a
function introduced by Zhikov \cite{Zhikov2000}, or is an eigenvalue
of the Dirichlet Laplacian in the inclusion $Q_0$ with a zero mean
eigenfunction. In the matrix, $u_\e\sim v^0(x)$, where $v^0$ is an
eigenfunction of the homogenized operator in $\Omega$, whose
eigenvalue $\nu$ determines the second term $\lm_1$ in the
asymptotics of $\lambda^\e$, see \eqref{lm1}. This is first derived
via formal asymptotic expansions, but then we prove a non-standard
error bound:
$$  | \lme - \lm_0 - \e \lm_1 | \le C \e^{5/4}, $$

 see Theorems \ref{t1} \& \ref{t2}.
The proof employs a combination of a high contrast boundary layer analysis with maximum principle and estimates in
Hilbert spaces with $\e$-dependent weights.
We finally briefly discuss further refinement of the results via the technique of two-scale convergence. Namely,
some version of the compactness result holds, cf. \cite{Zhikov2000}, indicating at the presence of gaps in the
spectrum for small enough $\e$, see Theorem \ref{2scalecomp}.

The paper is organized as follows.
The next section formulates the problem and introduces necessary notation,
Section 3 executes formal asymptotic expansion and derives associated homogenized equations. Section 4 proves
the error bounds and Section 5 discusses the two-scale convergence approach. Some technical details are assembled in
the appendices.

\section{Problem statement and notations}
We consider a model of eigenvibrations for a body occupying
a bounded domain $\Om$ in $\Real^n$ ($n=2,3,\dots$) containing a periodic array of
small inclusions, see Figure \ref{fig-02}.
The size of inclusions is controlled by a small positive parameter $\e$,
$\e\to 0$.
First we introduce necessary notation.

\begin{figure}[t]
\vspace{0.4cm}
\center
\includegraphics[width=0.7\textwidth]{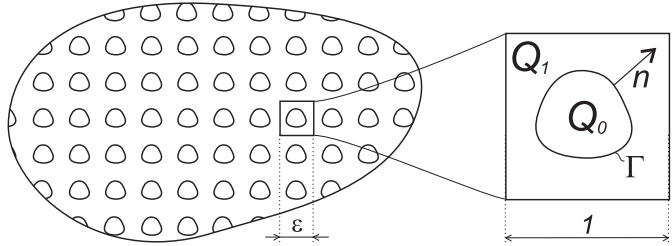}
\caption{The geometry and the periodicity cell}
\label{fig-02}
\end{figure}

Let $Q=[0,1]^n$ be a reference periodicity cell in $\Real^n$. Let
$\widetilde{Q}_0$ be a periodic set of ``inclusions'', i.e.
$\widetilde{Q}_0 + m = \widetilde{Q}_0$, $\forall m \in
\mathbb{Z}^n$, and $Q_0 = \widetilde{Q}_0 \cap Q$ is a reference
inclusion lying inside $Q$ with $C^2$-smooth boundary $\Gamma$, see
Figure \ref{fig-02}. Let $Q_1 = Q \backslash \overline{Q_0}$,
$\widetilde{Q}_1 = \Real^n \backslash \overline{\widetilde{Q}_0}$,
$\widetilde{\Gamma} = \p \widetilde{Q}_0 = \p \widetilde{Q}_1$.
Introducing $y=x / \e$ we refer to $y$ as to a fast variable, as
opposes to the slow variable $x$. In the $x$-variable the
periodicity cell is $\e Q=[0,\e)^n$. If $y\in Q_j$ then $x=\e y \in
\e Q_j$, $j=0,1$. We denote $\Om_0^\e:= \Om \cap \e\tilde{Q}_0$,
$\Om_1^\e:= \Om \cap \e\tilde{Q}_1=\Om\backslash
\overline{\Om_0^\e}$, $\Gamma^\e:=\e\tilde\Gamma\cap\Omega$, see
Figure \ref{fig-02}. The trace on $\tGe$ of function $f: \Om_j^\e
\to \Real^n$ is denoted by $f|_{j}$. Let $n_y$ be the outer unit
normal to $Q_0$ on its boundary $\Gamma$ and let $n_x$ denote the
similar normal on $\tGe$.

Let stiffness $a_\e$ and density $\rho_\e$ be as follows
\ben
a_\e
\left({x}\right)=\left\{\ba{rl}
1, & x\in \Om_1^\e\\
\e,& x\in \Om_0^\e
\ea\right.
\quad\mbox{and}\quad
\rho_\e\left({x}\right)=\left\{\ba{rl}
1, & x\in \Om_1^\e\\
\e^{-1},& x\in \Om_0^\e \ea\right.
\een
with a small positive $\e$.

We study the asymptotic behaviour of
self-adjoint spectral problem
\be
\intl_\Om
a_\e\left({x}\right) \nabla\ue \nabla\phi \, dx -
\lme \intl_\Om
\rho_\e\left({x}\right) \ue \phi \, dx = 0,
\quad \forall \phi\in
H^1_0(\Om) \label{variational}
\ee
as $\eto$. If $\Gamma$ and
$\p\Om$ are smooth enough then variational problem
\eqref{variational} can be equivalently represented in a classical
formulation
\bea
-\di(\aeps\left({x}\right) \nabla\ue)
&=& \lme \reps\left({x}\right) \ue,
\quad x\in\Om,
\label{divergent form}\\
\ue|_{\p\Om} &=& 0,
\label{boundary condition}
\eea
implying that at the interfaces
the transmission conditions are satisfied
\be
\ue\Big|_1 = \ue\Big|_0, \quad
\frac{\p \ue}{\p \nx}\Big|_1 = \e \frac{\p \ue}{\p \nx}\Big|_0.
\label{interfacial conditions}
\ee

\section{Formal asymptotic expansions}\label{fae}
We seek formal asymptotic expansions for the eigenvalues $\lme$ and
eigenfunctions $\ue$ in the form
\bea
\lme &\sim& \lm_0 + \e \lm_1 + \e^2 \lm_2 + \dots,
\label{lme: expansion}\\
\ue(x) &\sim&\left\{
\ba{ll}
\D v_0\left(x,\frac{x}{\e}\right) +
\e v_1\left(x,\frac{x}{\e}\right) +
\e^2v_2\left(x,\frac{x}{\e}\right)+ \dots, & x\in \Om_1^\e,\\[3mm]
\D w_0\left(x,\frac{x}{\e}\right) +
\e w_1\left(x,\frac{x}{\e}\right) +
\e^2w_2\left(x,\frac{x}{\e}\right)+
\dots, & x\in \Om_0^\e.
\ea
\right.
\label{ue: expansion}
\eea
Here all the functions $v_j(x,y)$, $w_j(x,y)$, $j\geq 0$, are required to be periodic in the ``fast'' variable $y$;
$v_0$ and $w_0$ are not simultaneously identically  zero
\be
 v_0^2 + w_0^2 \not \equiv 0.
\label{nonzer}
\ee

In a standard way, the ansatz \eqref{lme: expansion}, \eqref{ue: expansion} is then formally substituted into
\eqref{divergent form}--\eqref{interfacial conditions}. In particular, from
\eqref{divergent form}, for $(x,y)\in \Om\times Q_1$,  we obtain
\bea
-\Delta_y v_0 &=& 0,
\label{eq v0}\\
-\Delta_y v_1 &=& 2 \frac{\p^2 v_0}{\p x_j \p y_j},
\label{eq v1}\\
-\Delta_y v_2 &=& 2 \frac{\p^2 v_1}{\p x_j \p y_j} + \Delta_x v_0 + \lm_0 v_0,
\label{eq v2}
\eea
(with $\Delta_y$ and $\Delta_x$ denoting the Laplace operators in $y$ and $x$, respectively, and summation
henceforth implied with respect to repeated indices),
and for $(x,y)\in \Om\times Q_0$  we have
\bea
-\Delta_y w_0 &=& \lm_0 w_0,
\label{eq w0}\\
-\Delta_y w_1 &=& 2 \frac{\p^2 w_0}{\p x_j \p y_j} + \lm_1 w_0 +
\lm_0 w_1,
\label{eq w1}\\
-\Delta_y w_2 &=& 2 \frac{\p^2 w_1}{\p x_j \p y_j}
+\Delta_x w_0 + \lm_2 w_0 + \lm_1 w_1 + \lm_0 w_2.
\label{eq w2}
 \eea
Further, the first of conditions \eqref{interfacial conditions}
transforms to
\be
v_j(x,y)\yin=
w_j(x,y)\yin, \quad x\in\Om, \quad j=0,1\dots\, .
\label{bc w}
\ee
Similarly, the other transmission condition \eqref{interfacial conditions} yield
\bea
\frac{\p v_0}{\p \ny}\yin &=& 0,
\label{bc v0}\\
\frac{\p v_1}{\p \ny}\yin &=&
- \frac{\p v_0}{\p \nx}\yin + \frac{\p w_0}{\p \ny}\yin,
\label{bc v1}\\
\frac{\p v_2}{\p \ny}\yin &=&
- \frac{\p v_1}{\p \nx}\yin +
\frac{\p w_1}{\p \ny}\yin + \frac{\p w_0}{\p \nx}\yin.
\label{bc v2}
\eea
%
The above has employed the identity
\be
\frac{\p u}{\p \nx} \left(x,\frac{x}{\e}\right)
\,=\,
\e^{-1} \frac{\p u}{\p \ny} (x,y) + \frac{\p u}{\p \nx} (x,y),
\quad y = \frac{x}{\e},
\label{two scale nabla}
\ee
where
$\frac{\p }{\p \ny} := n_y \cdot \nabla_y$,
$\frac{\p }{\p \nx} := n_y \cdot \nabla_x$,
with $\nabla_y$ and $\nabla_x$ standing for gradients in $y$ and $x$, respectively.

Finally, \eqref{boundary condition} suggests
\be
v_0\Big|_{x\in\p\Om}=w_0\Big|_{x\in\p\Om}=0.
\ee
(The boundary layer problem does not generally permit satisfying \eqref{boundary condition}
by $v_j$ and $w_j$ for $j\geq 1$, as also clarified later.)

Combining \eqref{eq v0} and \eqref{bc v0},
together with the periodicity conditions in $y$, implies that $v_0$ is a constant
with respect to $y$, i.e.
$$
v_0(x,y)\equiv v^0(x).
$$
Then, \eqref{eq v1} and \eqref{bc v1} form the following boundary value problem for $v_1$
\bea
\qquad
-\Delta_y v_1(x,y) = 0
\quad\mbox{in}\quad  \Om\times Q_1, \qquad
\frac{\p v_1}{\p \ny}\yin =
- \frac{\p v^0}{\p \nx}\yin + \frac{\p w_0}{\p \ny}\yin.
\label{problem v1}
\eea
The latter is solvable if and only if
\be
\intl_{\tG} \frac{\p w_0}{\p \ny} \, dy =0.
\label{orthogonality w0}
\ee

Considering next \eqref{eq w0} and \eqref{bc w} gives
\be
-\Delta_y w_0 = \lm_0 w_0
\quad\mbox{in}\quad  \Om\times Q_0, \qquad
w_0(x,y)\yin = v^0(x).
\label{problem w0}
\ee
Since
$$
\intl_{\tG} \frac{\p w_0}{\p \ny} \, dy =
\intl_{Q_0} \Delta_y w_0 \, dy = - \lm_0 \intl_{Q_0} w_0 \, dy,
$$
condition \eqref{orthogonality w0} is equivalent to
\be
\lm_0 \langle w_0 \rangle = 0,
\label{orthogonality w0 2}
\ee
where
$$
\langle u \rangle :=  \intl_{Q_0} u(y) \, dy.\qquad
$$

We notice that \eqref{problem w0}--\eqref{orthogonality w0 2}
together with \eqref{nonzer} constitutes restrictions on possible
values of $\lambda_0$.
Those are described by Theorem \ref{limoper1} below.
Before, let us consider an auxiliary Dirichlet problem
\be
-\Delta_y \phi = \lmd \phi
\quad\mbox{in}\quad  Q_0, \qquad
\phi\Big|_{\Gamma} = 0.
\label{auxiliary phi}
\ee
Let $\{\lmd_j\}_{j=1}^\infty$ be eigenvalues for \eqref{auxiliary phi}, labelled in the
ascending order counting for the multiplicities, and let $\{\phi_j\}_{j=1}^\infty$
be the corresponding  eigenfunctions, orthonormal in $L_2(Q_0)$, i.e.
$$
\intl_{Q_0} \phi_j\phi_k \,dy=\delta_{jk},
$$
where $\delta_{jk}$ is Kronecker's delta. Denote by $\sigma_D$ the spectrum of
\eqref{auxiliary phi}: $\sigma_D=\bigcup_{j=1}^\infty \lmd_j$.

We additionally introduce the following auxiliary problem:
\be
-\Delta_y \eta = \lm_0 \eta
\quad\mbox{in}\quad   Q_0, \,\, \quad  \eta(y)\yin = 1.
\label{probleta}
\ee
Notice that \eqref{probleta} is solvable if and only if $\lambda_0\not\in\sigma_D$ or
$\lambda_0=\lambda_j^D$ with all the associated eigenfunctions $\phi_j$ having zero mean,
$\langle\phi_j\rangle=0$\footnote{We remark that the case of eigenvalues with zero mean is known to be
not a ``generic'' case,
i.e. unstable via a small perturbation of the shape of $Q_0$, see e.g. discussion in
\cite{hempel} and
further references therein.}. In the former case $\eta$ is determined uniquely and \eqref{problem w0} implies
$w_0(x,y)=v^0(x)\eta(y)$.
In the latter case $\eta$ is determined up to
an arbitrary eigenfunction $\phi_j$ associated with $\lambda_j^D$, however $\langle\eta\rangle$
is determined uniquely.

By direct inspection, \eqref{problem w0}, \eqref{orthogonality w0 2}
has a non-trivial solution $(v^0, w_0)$,
i.e. with \eqref{nonzer} holding,
if and only if $\lm_0$
is an eigenvalue of following problem:
\be
-\Delta_y \zeta = \lm_0 \zeta
\quad\mbox{in}\quad   Q_0, \qquad
\zeta(y)\yin = \mbox{constant}, \qquad \lambda_0\langle\zeta\rangle=0.
\label{probzeta}
\ee

\begin{theorem}
\label{limoper1}
The problem \eqref{probzeta} is equivalent to an eigenvalue problem for
a self-adjoint operator in $L_2(Q_0)$ with a compact resolvent.
Therefore the spectrum of \eqref{probzeta} is a countable set
of real non-negative eigenvalues (of finite multiplicity)
with the only accumulation point at $+\infty$, with the eigenfunctions complete in
$L_2(Q_0)$ and
those corresponding to different $\lambda_0$ mutually orthogonal.

The spectrum consists of all
the eigenvalues $\lmd$ of problem \eqref{auxiliary phi}
with a zero mean eigenfunction and all the
solutions of the equation
\be
 B(\lambda_0)\,:= \,\lambda_0\langle\eta\rangle=
 \lm_0 \left( |Q_0| + \lm_0
\sum_{j=1}^\infty \frac{ \langle \phi_j \rangle^2}{\lmd_j-\lm_0} \right) = 0
\label{bkslambda}
\ee
(which are hence all real non-negative). In \eqref{bkslambda} the summation is with
respect to only those $\lambda_j^D$ for which there exists an eigenfunction with a non-zero mean.

The associated eigenfunctions $\zeta$ are either proportional to $\eta$ as in \eqref{probleta} or are
eigenfunctions of \eqref{auxiliary phi} with zero mean.
\end{theorem}

\bpr
 We claim that
\eqref{probzeta} corresponds to a self-adjoint operator associated with the
(symmetric, closed, densely defined, bounded from below) Dirichlet form
\be
\alpha (\zeta, h)\,:=\,\int_{Q_0}\nabla\zeta\cdot\nabla h \,\, dy
\label{dirichl}
\ee
with domain
\be
D(\alpha):=\{h\in H^1(Q_0):\quad h\yin=\mbox{ constant}\}.
\label{dom}
\ee
To see this, in the weak formulation of the eigenvalue problem associated with
\eqref{dirichl}--\eqref{dom}
\be
\int_{Q_0}\nabla\zeta\cdot\nabla h\,\, dy \,
= \,\lambda_0 \int_{Q_0} \zeta\, h \,\, dy,
\qquad \forall h\in D(\alpha),
\label{wf}
\ee
we first set $h$ to be an arbitrary function from $C^\infty_0(Q_0)$ which implies
$-\Delta_y\zeta=\lambda_0\zeta$ in $Q_0$, and then set $h\equiv 1$ yielding $\lambda_0\langle\zeta\rangle=0$.
Further, since the resolvent is obviously compact,
each eigenvalue has a finite multiplicity,
the set of all eigenfunctions $\zeta$ is complete in $L_2(Q_0)$ and
those corresponding to different $\lambda_0$ are mutually orthogonal.

Obviously, the spectrum of \eqref{probzeta} includes
those and only those eigenvalues of \eqref{auxiliary phi} which have an
eigenfunction $\phi_j$ with zero mean.
In this case corresponding eigenfunctions of  \eqref{probzeta}
are given by
$\zeta_j = C \phi_j$, $C \not= 0$. If $\lambda_j^D$ does not have a
zero-mean eigenfunction, then the solvability of \eqref{probzeta}
requires $\zeta\yin = 0$ implying $\zeta\equiv 0$.
Considering other possibilities, fix $\lambda_0$
outside $\sigma_D$ and
let $\eta$ be the unique solution of \eqref{probleta}.
Then $\lambda_0$ is an eigenvalue of \eqref{probzeta}
if and only if
\be
\lambda_0\langle\eta\rangle \,=\, 0,
\label{probleta2}
\ee
with corresponding eigenfunction given by
$\zeta(y) = C \eta(y)$, $C \not= 0$.

Via the spectral decomposition, the solution to
\eqref{probleta} is found to be, cf. \cite{Zhikov2000}:
\be
\eta(y)=1+
\lm_0 \sum_{j=1}^\infty \frac{ \langle \phi_j \rangle}{\lmd_j-\lm_0} \phi_j(y).
\label{spectral decomposition w0}
\ee
Substituting \eqref{spectral decomposition w0} further into \eqref{probleta2}
yields \eqref{bkslambda}.
\epr

The formula \eqref{bkslambda} can be transformed to read
\be
 B(\lm_0) = \beta(\lm_0) - |Q_1| \lm_0\, = 0,
\label{equation for lm0}
\ee
where function $\beta(\lm)$ has been introduced by
Zhikov \cite{Zhikov2000}:
\be
\beta(\lm) =
\lm + \lm^2 \sum_{j=1}^\infty \frac{ \langle \phi_j \rangle^2}{\lmd_j-\lm},
\label{betazhik}
\ee
see Figure \ref{fig-1}.
This implies that $\lm_0$ is either a solution to the nonlinear equation
\be
\beta(\lm) = |Q_1| \lm,
\label{equation for lm0 2}
\ee
as visualized on Figure \ref{fig-1}, or is an eigenvalue of \eqref{auxiliary phi} with a zero
mean eigenfunction.

\begin{figure}[t]
\vspace{0.4cm}
\center\includegraphics[width=0.7\textwidth]{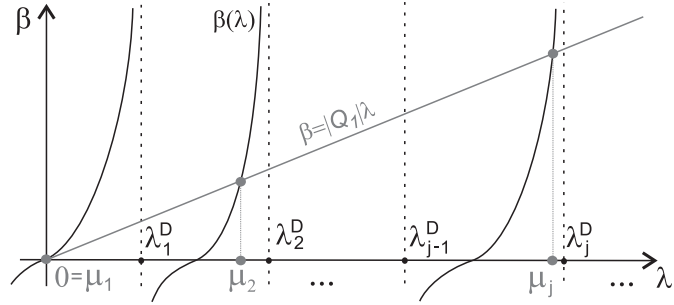}
\caption{The limit eigenvalues $\lm_0=\mu_j$}
\label{fig-1}
\end{figure}

\begin{remark}
\emph{ If $Q_0$ is a ball of radius $0<a<1/2$, i.e   $Q_0=B_a=\{y:|y|<a\}+y_0$, then we have an explicit
representation for $\beta(\lm)$. Indeed, for $\lambda_0\not\in\sigma_D$ the solution of
\eqref{probleta} is radially symmetric and (placing the origin in the ball's centre) reads
\ben
        \eta(y)=|y|^{\frac{2-n}{2}}J_{\frac{n-2}{2}}(\lm_0^{1/2}|y|)
        \left(|a|^{\frac{2-n}{2}}J_{\frac{n-2}{2}}(\lm_0^{1/2}a)\right)^{-1},
\een
where  $J_{\frac{n-2}{2}}(|y|)$  is  Bessel function.  Further, we
have
\ben
     B(\lambda_0) =\lambda_0  \langle\eta\rangle
     =
-\int_{\p B_a}{\frac{\p \eta}{\p n_y}} dy=
\een
\ben
        =-\frac{1}{a }|\Gamma|
        \left(1-n/2+
        a\lm_0^{1/2}J_{\frac{n-2}{2}}'(\lm_0^{1/2}a)/J_{\frac{n-2}{2}}(\lm_0^{1/2}a)\right)
        .
\een
Using \eqref{betazhik}, \eqref{spectral decomposition w0} we obtain
\ben
        \beta(\lambda)=\lm(1-|B_a|)-\frac{1}{a }|\Gamma|
        \left(1-n/2+
        a\lm^{1/2}J_{\frac{n-2}{2}}'(\lm^{1/2}a)/J_{\frac{n-2}{2}}(\lm^{1/2}a)\right)
        .
\een
In particular, for $n=3$ we have,
\ben
  B(\lambda_0) = \lambda_0      \langle\eta\rangle=
        {4\pi a}
        \left(1-
        a\lm_0^{1/2} \, \emph{cotan} \,(\lm_0^{1/2}a)\right)
        ,
\een
\ben
    \beta(\lambda)=\lm(1-4\pi a^3/3)+4\pi a
        \left(1-
        a\lm^{1/2} \, \emph{cotan} \,(\lm^{1/2}a)\right)
        .
\een
}\end{remark}

\medskip

We next explore in detail the further steps in the method of asymptotic expansions, to determine
$v^0$, etc.
Let us consider a $K$-dimensional eigenspace ($K \ge 1$)
for a given eigenvalue $\lm_0$
of \eqref{probzeta}, and
let  $\zeta_1, \dots, \zeta_K$ be associated linearly
independent eigenfunctions.
Then, \eqref{problem w0} and \eqref{orthogonality w0 2}
imply
\be
w_0(x,y) = \sum_{k=1}^K c_k(x) \zeta_k(y).
\label{w0 representation}
\ee
Following Theorem \ref{limoper1} we distinguish  two cases:
\begin{itemize}

\item[(a)]
$\lambda_0 \not\in \sigma_D$. In this case \eqref{probleta}
and \eqref{problem w0} suggest
\be
w_0(x,y)=v^0(x)\eta(y),
\label{case1w0}
\ee
and \eqref{nonzer} implies $v^0 \not\equiv 0$.

\item[(b)]
$\lambda_0 \in \sigma_D$.
The latter means
$\lambda_0=\lmd_j$ for some $j$.
This includes two further possibilities:
\begin{itemize}
\item[(i)] The eigenspace of \eqref{auxiliary phi}
has an eigenfunction $\phi_j^*$ with a non-zero mean.
Since the solvability conditions for \eqref{problem w0} include
\be
v^0(x)\langle \phi_j^* \rangle\,=\,0,
\label{case2solv}
\ee
necessarily $v^0 \equiv 0$.
Moreover, with $K_D$ denoting the multiplicity of $\lambda_j^D$ as of the eigenvalue of the Dirichlet
problem \eqref{auxiliary phi}, necessarily $K_D \ge 2$: if $K_D=1$ then
$w_0 = C(x) \phi_j^*$ and thus \eqref{orthogonality w0 2}
implies $C(x) \equiv  0$ and $w_0 \equiv 0$ contradicting to \eqref{nonzer}.
Hence $w_0$ is given by  \eqref{w0 representation} with $K=K_D-1$, with $\zeta_k$, $k=1,...,K$ being
linearly independent eigenfunctions of \eqref{auxiliary phi} with zero mean (such $K$ eigenfunctions
exist).
\item[(ii)]
All of the eigenfunctions corresponding $\lmd_j$ have a zero mean.
In this case $w_0$ is again given by \eqref{w0 representation}, with $K=K_D$ if $\langle\eta\rangle\neq0$ i.e. $B(\lambda_0)\neq 0$ and
$K=K_D+1$ if $B(\lambda_0)=0$ with $\zeta_{K_D+1}=\eta$ where $\eta(y)$ is any solution of \eqref{probleta}.
\end{itemize}
\end{itemize}

\subsection[Case (a)]{Case (a): \protect{$\lambda_0 \not\in \sigma_{D}$}}
\label{formal case a}
In this case $\lambda_0$ are solutions of \eqref{equation for lm0 2}.
There is a countable set of $\lm_0=\mu_j$, $j=1,2,\dots$ as Figure \ref{fig-1} illustrates.
Note that this includes $\lm_0=0$.
Function $\beta$ blows up at the points $\lmd_j$,
which are eigenvalues of \eqref{auxiliary phi} having an eigenfunction with a non-zero mean,
monotonically increasing between such points.
It also directly follows from \eqref{betazhik} that $\beta(\lambda) > |Q_1|\lambda$ for $\lambda\in (0, \lmd_1)$,
implying  $\lmd_1 < \mu_2 < \lmd_2$.
Let  $\lm_0$ satisfying \eqref{equation for lm0 2} be fixed.

We consider problem \eqref{problem v1} taking into account \eqref{case1w0}, i.e.
\be
-\Delta_y v_1(x,y) = 0
\quad\mbox{in}\quad  \Om\times\cQ_1, \qquad
\frac{\p v_1}{\p \ny}\yin =
- \frac{\p v^0}{\p \nx}\yin + v^0(x)\frac{\p \eta}{\p \ny}\yin,
\label{probl v1}
\ee
where $\eta(y)$ solves \eqref{probleta} and is given by \eqref{spectral decomposition w0}.
Hence $v_1$ is a solution to a problem
depending linearly on $v^0$ and $\nabla_x v^0$,
implying
\be
v_1(x,y) = v^0(x) \cN(y) + \frac{\p v^0}{\p x_j} N_j(y)
+ v_1^*(x),
\label{representation v1}
\ee
with an arbitrary function $ v_1^*(x)$. The choice of $v_1^*$ does not affect the subsequent
constructions, so we set for simplicity $v_1^* \equiv 0$.
In \eqref{representation v1}
functions $N_j$ and $\cN$ are solutions to the problems
\be
\Delta_y N_j(y)= 0 \quad\mbox{in}\quad Q_1,\qquad
\frac{\p N_j}{\p n_y}\Big|_{y\in\Gamma} = - n_j(y),
\label{problem Nj}
\ee
and
\be
\Delta_y \cN(y) = 0 \quad\mbox{in}\quad Q_1,\qquad
\frac{\p \cN}{\p n_y}\Big|_{y\in\Gamma} = \frac{\p \eta}{\p n_y}\Big|_{y\in\Gamma}.
\label{problem cN}
\ee
Solvability of \eqref{problem cN} requires
$$
\intl_{\Gamma} \frac{\p \eta}{\p n_y} \, dy = 0,
$$
which is equivalent to \eqref{probleta2} and is hence already assured.
Since the solutions of \eqref{problem Nj} and \eqref{problem cN} are unique up to an
arbitrary constant, we fix those by choosing
$$
\intl_{Q_1} N_j(y) \, dy\, =\,
\intl_{Q_1} \cN(y) \, dy = 0.
$$

We next consider
the problem for $w_1$, which from \eqref{eq w1} and \eqref{bc w} combined with
\eqref{case1w0} reads
\begin{align}
-\Delta_y w_1 - \lm_0 w_1
&=
\lm_1 v^0 \eta + 2 \D\frac{\p v^0}{\p x_j} \frac{\p \eta}{\p y_j}
\quad\mbox{in}\quad \Om\times\cQ_0,
\label{problem w1 eq}
\\
w_1 \yin
&=
\D v^0 \cN(y)\yin + \frac{\p v^0}{\p x_j} N_j(y)\yin.
\label{problem w1 bc}
\end{align}
Since the problem depends linearly on $v^0$, $\lm_1 v^0$ and $\frac{\p v^0}{\p x_j}$,
the solution admits representation
\be
w_1(x,y) = \frac{\p v^0}{\p x_j}(x) \cM_j(y) + v^0(x) \cP(y) + \lm_1 v^0(x) \cR(y),
\label{representation w1}
\ee
where functions $\cM_j$, $\cP$ and $\cR$ are solutions to the problems
\be
-\Delta_y \cM_j - \lm_0 \cM_j = 2 \frac{\p \eta}{\p y_j}(y)
\quad\mbox{in}\quad Q_0,
\qquad
\cM_j\Big|_{\Gamma} = N_j\Big|_{\Gamma},
\label{problem cM}
\ee
\be
-\Delta_y \cP - \lm_0 \cP = 0
\quad\mbox{in}\quad Q_0,
\qquad
\cP\Big|_{\Gamma} = \cN\Big|_{\Gamma},
\label{problem cP}
\ee
and
\be
-\Delta_y \cR - \lm_0 \cR = \eta(y)
\quad\mbox{in}\quad Q_0,
\qquad
\cR\Big|_{\Gamma} = 0.
\label{problem cR}
\ee
Since by the assumption $\lambda_0\not\in\sigma_D$,
all the problems \eqref{problem cM} -- \eqref{problem cR} are uniquely solvable.

The problem for $v_2$ is in turn given by
\eqref{eq v2} and \eqref{bc v2},
whose solvability condition hence reads
\be
\intl_{Q_1} \left( \Delta_x v_0 + \lm_0 v_0 + 2 \frac{\p^2 v_1}{\p x_j \p y_j} \right)\, dy =
\intl_{\Gamma} \left( - \frac{\p v_1}{\p \nx} +
\frac{\p w_1}{\p \ny} + \frac{\p w_0}{\p \nx} \right)\, dy,
\label{solvability v2}
\ee
with functions $v_1$, $w_1$ and $w_0$ given by \eqref{representation v1},
\eqref{representation w1} and \eqref{case1w0} respectively.

Appendix A provides a detailed calculation showing that the above yields the following equations
for $v^0$:
\bea
-\di \Ah \nabla_x v^0 &=& \nu(\lambda_1) v^0 \quad\mbox{in}\quad \Om,
\label{homogenized problem eq}\\
v^0\Big|_{\p\Om} &=& 0.
\label{homogenized problem bc}
\eea
Here $\Ah=\left( \Ah_{jk}\right)_{j,k=1}^n$ is the classical homogenized matrix for
periodic perforated domains, see e.g. \cite{JKO}
\be
\Ah_{jk} =|Q_1|\delta_{jk} + \intl_{Q_1} \frac{\p N_k}{\p y_j} \, dy;
\label{A hom}
\ee
\be
\nu(\lm_1) = \cC \lm_1+ \lm_0 \Big( |Q_1| + \intl_{Q_0} \cP \, dy \Big),
\label{nulam1}
\ee
where
\be
\cC:=\int_{Q_0}\eta^2dy > 0.
\label{cgarnaja}
\ee
Note that the problem \eqref{homogenized problem eq}--\eqref{homogenized problem bc}
involves $\nu = \nu(\lm_1)$ as a spectral parameter.

The spectrum of \eqref{homogenized problem eq}--\eqref{homogenized problem bc} consists of a
countable set of eigenvalues
\be
0 < \nu_1 < \nu_2 \le \dots \le \nu_n \le \dots \to +\infty.
\label{nun}
\ee
Corresponding eigenfunctions $v_n$ form an orthonormal basis
in $L_2(\Omega)$,
$$
\intl_{\Om} v^0_n v^0_m \, dx = \delta_{nm}.
$$

Fixing an eigenvalue $\nu$ of \eqref{homogenized problem eq},
\eqref{homogenized problem bc} with corresponding
eigenfunction $v^0$ of unit norm in $L_2(\Om)$,
according to \eqref{nulam1} we find
\be
\lm_1 = \cC^{-1} \left(
\nu - \lm_0 \Big( |Q_1| + \intl_{Q_0} \cP \, dy \Big)
\right).
\label{lm1}
\ee

The following diagram summarizes the algorithm for
constructing the first terms of the asymptotic expansions
(for the case $\lambda_0\not\in\sigma_D$)
$$
\left.
\ba{lcrcl}
N_j &\stackrel{\eqref{A hom}}{\rightarrow}& \Ah
&\stackrel{\eqref{homogenized problem eq}}{\rightarrow}& \nu, v^0\\
\lm_0 &\stackrel{\eqref{probleta}}{\rightarrow}& \eta &\stackrel{\eqref{problem cN}}{\rightarrow}& \cN
\ea
\right\}
\left.
\ba{ll}
\stackrel{\eqref{lm1}}{\rightarrow}& \lm_1\\
\stackrel{\eqref{representation v1}}{\rightarrow}& v_1\\
\stackrel{\eqref{problem cM}-\eqref{problem cR}}{\rightarrow}& \cM, \cP, \cR\\
\stackrel{\eqref{case1w0}}{\rightarrow}& w_0
\ea
\right\}
\stackrel{\eqref{representation w1}}{\rightarrow} w_1
\stackrel{\eqref{eq v2},\eqref{bc v2}}{\rightarrow} v_2.
$$

We can additionally construct $w_2$ from \eqref{eq w2} and \eqref{bc w},
whose unique solution exists for any choice of $\lambda_2$.
For purposes of  the justification of the first two terms in the asymptotics (the next section)
it is sufficient to set $\lm_2 = 0$ and fix the corresponding solution $w_2$.

This completes constructing a formal asymptotic approximation, which we now summarize.
We introduce an approximate eigenvalue
\be
\Lambda_\e = \lm_0 + \e \lm_1,
\label{Lme}
\ee
and corresponding approximate eigenfunction
\be
\We(x)=
\left\{
\ba{ll}
\D v^0(x) + \e v_1\left(x,\frac{x}{\e}\right)
+ \e^2 v_2\left(x,\frac{x}{\e}\right)
,
& x \in \Om_1^\e,
\\
\D w_0\left(x,\frac{x}{\e}\right) + \e
w_1\left(x,\frac{x}{\e}\right) + \e^2 w_2\left(x,\frac{x}{\e}\right)
, & x \in \Om_0^\e. \ea \right. \label{Ue} \ee

The essence of the above formal asymptotic construction is that the
action of differential operator ${\Ae}$ on $\We$ defined by
\be
{\Ae}\We\,:=\,   \di(
a_\e \nabla \We )
+\Lambda_\e
\rho_\e \We
\label{leps}
\ee
produces a small right-hand side in both $\Omega_1^\e$ and $\Omega_0^\e$, and
on the interface $\Gamma^\e$ in the following sense.
\bl
$
(i) \,\,\,
\D \max_{\bar{\Om}^\e_1} \left|
\emph{\di} (
a_\e \nabla \We
)
+
\Lambda_\e
\rho_\e \We
\right|
\le
C \e.
$

$
(ii) \,\,\,
\D \max_{\bar{\Om}^\e_0} \left|
\emph{\di} (
a_\e \nabla \We
)
+
\Lambda_\e
\rho_\e \We
\right|
\le
C \e^2.
$

$
(iii) \,\,\,
\D \max_{\tGe} \left|
\left. a_\e \frac{\p \We}{\p n} \right|_0 -
\left. a_\e \frac{\p \We}{\p n} \right|_1
\right|
\le C \e^2.$
\label{estimate in Ome1}
\el
\bpr
$(i)\,$
Since the function $\We$ is two-scale by the construction, in $\Ome_1$
{\small
\begin{align}
\di(&
a_\e \nabla \We )
+
\Lambda_\e
\rho_\e \We
=
\n \\
& =
\left(
\e^{-2} \Delta_y +
\e^{-1} 2 \frac{\p^2}{\p x_j \p y_j} +
\Delta_x
+ \lm_0 + \e \lm_1
\right)
(v^0(x) + \e v_1(x,y) + \e^2 v_2(x,y))|_{y=\frac{x}{\e}}
=
\n \\
& =
\left\{
\e^{-1} \Delta_y v_1(x,y) +
\e^{0} \left(
\Delta_y v_2 +
2 \frac{\p^2 v_1}{\p x_j \p y_j} +
\Delta_x v^0 + \lm_0 v^0
\right)
+
\right.
\n \\
& +
\left.\left.
\e^1 \left(
2 \frac{\p^2 v_2}{\p x_j \p y_j} +
\Delta_x v_1 + \lm_1 v^0 + \lm_0 v_1
\right)
+
\e^2 (\Delta_x v_2 + \lm_1 v_1 +\lm_0 v_2)
+
\e^3 \lm_1 v_2
\right\}\right|_{y=\frac{x}{\e}}.
\label{coef estimate 1}
\end{align}
}
Since $v_1$ is a solution to \eqref{probl v1},
the coefficient of $\e^{-1}$ vanishes.
The same is with the coefficient of $\e^0$
since $v_2$ satisfies \eqref{eq v2}.
Functions $v^0$, $v_1$ and $v_2$
are solutions of elliptic problems with smooth enough
coefficients to guarantee belonging
solutions to $C^2$. Thus,  maxima for coefficients
of $\e^1$, $\e^2$ and $\e^3$ in \eqref{coef estimate 1} exist.

$(ii)\,$
Similarly, in $\Ome_0$
{\small
\begin{align}
&\di(
a_\e \nabla \We )
+
\Lambda_\e
\rho_\e \We
=
\n \\
& =
\left(
\e^{-1} \Delta_y + 2 \frac{\p^2}{\p x_j \p y_j} + \e \Delta_x
+ \e^{-1} \lm_0 + \lm_1
\right)
(w_0(x,y) + \e w_1(x,y) + \e^2 w_2(x,y))|_{y=\frac{x}{\e}}
=
\n \\
& =
\left\{
\e^{-1} ( \Delta_y w_0 + \lm_0 w_0 ) +
\e^{0} \left(
 \Delta_y w_1 + 2 \frac{\p^2 w_0}{\p x_j \p y_j} +
\lm_0 w_1 + \lm_1 w_0
\right)
+
\right.
\n \\
& +
\e^1 \left(
\Delta_y w_2 +
2 \frac{\p^2 w_1}{\p x_j \p y_j} +
\Delta_x w_0 + \lm_0 w_2 + \lm_1 w_1
\right)
+
\n \\
& +
\left.\left.
\e^2 \left(
2 \frac{\p^2 w_2}{\p x_j \p y_j} +
\Delta_x w_1 + \lm_1 w_2
\right)
+
\e^3
\Delta_x w_2
\right\}\right|_{y=\frac{x}{\e}}.
\n
\end{align}
}
Since $w_0(x,y) = v^0(x) \eta(y)$ is chosen according to
\eqref{probleta}, the coefficient of $\e^{-1}$ vanishes.
The coefficient of $\e^0$  vanishes due to
\eqref{problem w1 eq}.
Further,
$w_2$ satisfies \eqref{eq w2} with $\lm_2 = 0$
and thus the coefficient of $\e^1$ is zero as well.
Since $w_1$ and $w_2$
are solutions of elliptic problems with smooth enough
coefficients,
the maxima of the coefficients of  $\e^2$ and $\e^3$ exist.

$(iii)\,$
Using \eqref{two scale nabla}, we obtain
{\small
\begin{align}
\left. a_\e \frac{\p \We}{\p n} \right|_0
& -
\left. a_\e \frac{\p \We}{\p n} \right|_1
=
\left(
\frac{\p}{\p \ny} + \e \frac{\p}{\p \nx}
\right)
(w_0(x,y) + \e w_1(x,y) +\e^2 w_2(x,y) )
\Big|_{\substack{x\in \tGe \\
y \in \tG \hspace{1.5mm}}}
-
\n \\
& -
\left(
\e^{-1}\frac{\p}{\p \ny} + \frac{\p}{\p \nx}
\right)
(v^0(x) + \e v_1(x,y) +\e^2 v_2(x,y) )
\Big|_{\substack{x\in \tGe \\
y \in \tG \hspace{1.5mm}}}
=
\n \\
& = \e^0 \left( \frac{\p w_0}{\p \ny} - \frac{\p v_1}{\p \ny} -
\frac{\p v^0}{\p \nx} \right)
\Big|_{\substack{x\in \tGe \\
y \in \tG \hspace{1.5mm}}}
+
\e^1 \left(
\frac{\p w_1}{\p \ny}
+
\frac{\p w_0}{\p \nx}
-
\frac{\p v_2}{\p \ny}
-
\frac{\p v_1}{\p \nx}
\right)
\Big|_{\substack{x\in \tGe \\
y \in \tG \hspace{1.5mm}}}
+
\n \\
& +
\e^2 \left(
\frac{\p w_2}{\p \ny}
+
\frac{\p w_1}{\p \nx}
-
\frac{\p v_2}{\p \nx}
\right)
\Big|_{\substack{x\in \tGe \\
y \in \tG \hspace{1.5mm}}}
+
\e^3
\frac{\p w_2}{\p \nx}
\Big|_{\substack{x\in \tGe \\
y \in \tG \hspace{1.5mm}}}.
\end{align}
}
The coefficients of $\e^{0}$ and $\e^1$ vanish because of
\eqref{problem v1} and \eqref{bc v2} respectively. The
rest of the coefficients are smooth enough to guarantee that their
maxima for $x\in \tGe$ and $y \in \Gamma$ exist. \epr

\subsection[Case (b)]{Case (b): $\lambda_0=\lambda_j^D$}
\label{section: formal case b}
For simplicity, we consider here only the case of eigenvalues of multiplicity
$K=1$ with zero mean eigenfunction ($\phi=\phi_j$), assuming additionally $\lambda_0$ is
{\it not} a solution of \eqref{equation for lm0}. All other degenerate cases,
see page \pageref{case2solv}, could be considered similarly.

In this case we can introduce a refined approximation for the eigenfunction
    \be
    \We^*(x)= \left\{ \ba{ll} \D \e v_1\left(x,\frac{x}{\e}\right) +\e^2
    v_2\left(x,\frac{x}{\e}\right), & x \in \Om_1^\e,
    \\
    \D w_0\left(x,\frac{x}{\e}\right) + \e
    w_1\left(x,\frac{x}{\e}\right) +\e^2 w_2\left(x,\frac{x}{\e}\right),
    & x \in \Om_0^\e. \ea \right.
     \label{w'}
     \ee
where
    \be
        w_0(x,y)=c(x)\phi(y).
    \ee

\bl
Let $c\in C^3(\Omega)$, then
there exist smooth functions
$v_1,w_1,v_2,w_2$ and a constant $\lambda_1$ such that
$\Lambda_\e = \lambda_j^D + \e \lm_1$
and $\We^*$ defined by \eqref{w'} satisfy
\begin{itemize}

\item[(i)]
\quad
$ \We^*(x)\in C(\Omega),$

\item[(ii)]
\quad
$  \D \max_{\bar{\Om}^\e_1} \left| \emph{\di} ( a_\e
\nabla \We^* ) + \Lambda_\e \rho_\e \We^* \right| \le C \e, $

\item[(iii)]
\quad
$  \D \max_{\bar{\Om}^\e_0} \left| \emph{\di} ( a_\e
\nabla \We^* ) + \Lambda_\e \rho_\e \We^* \right| \le C \e^2, $

\item[(iv)]
\quad
$ \D \max_{\tGe} \left| \left. a_\e \frac{\p \We^*}{\p n}
\right|_0 - \left. a_\e \frac{\p \We^*}{\p n} \right|_1 \right| \le C
\e^2.$
\end{itemize}
\label{est p}
\el

\bpr
See Appendix B.
\epr

\section{Justification of asymptotics}

\subsection{Operator formulation}
We use a standard notation for Lebesgue and Sobolev spaces:
$L^2_p(\Om)$ is a $p$-weighted $L^2$-space of square-integrable
functions in $\Om$.
Notation $(\cdot,\cdot)_H$ is used
for a scalar product in a Hilbert space $H$.

Let $\Le = L^2_{\rho_\e}(\Om)$ and
$\He$ be $H^1_0(\Om)$ Sobolev space with a scalar product
$$
(u,v)_\He = \intl_\Om a_\e(x) \nabla u \cdot \nabla v \, dx +
\intl_\Om \rho_\e(x) u v \, dx.
$$
Following a standard procedure, see e.g. \cite{JKO}, we introduce a
bounded operator $\Be : \Le \to \Le$ such that
    \be
        (\Be f, v)_\He =
        (f,v)_\Le, \quad \forall v \in \He.
    \label{operator definition}
    \ee
In other words $\Be f=\ue$, where $\ue$ is the solution of the
problem
    \bea
        -\di(\aeps \nabla\ue) +\reps \ue&=&
        \reps f, \quad x\in\Om,
        \label{divergent form4}\\
        \ue|_{\p\Om} &=& 0,
    \label{boundary condition4}
    \eea
 \be \ue\Big|_1 = \ue\Big|_0, \quad \frac{\p \ue}{\p
\nx}\Big|_1 = \e \frac{\p \ue}{\p \nx}\Big|_0. \label{interfacial
conditions4} \ee
 Note that operator
$\Be$ is positive, self-adjoint and compact for any fixed $\e > 0$
(since its image is in $\He$).
Eigenvalue problem \eqref{variational}
is equivalent to \be \Be \ue = (\lme + 1)^{-1} \ue \quad \text{in}
\quad \Le. \label{operator formulation} \ee Hence the spectrum of the
problem consists of a countable set of eigenvalues
$$
0 < \lme_1 < \lme_2 \le \dots \le \lme_k \le \dots \to +\infty,
$$
with the only accumulation point at $+\infty$. Moreover, the set of
corresponding eigenfunctions is complete in $\Le$.

\subsection{Case (a)}
In this Section we justify the leading terms of asymptotic expansions
constructed above in case $\lm_0 \not\in \sigma_D$
and thus $v^0 \not\equiv 0$, see Section \ref{formal case a}.
Let $\lm_0$ be a solution to equation \eqref{equation for lm0 2}.
All the functions
($\eta$, $N_j$, $\cN$, $\cM$, $\cP$, $\cR$, $w_0$, $w_1$, $v_1$
and $w_2$, $v_2$)
are as defined in Section \ref{formal case a}.
We also fix $\lm_1$ according to \eqref{lm1}.
The approximate eigenvalue $\Lambda_\e$ and eigenfunction $W_\e$ are given
by \eqref{Lme} and \eqref{Ue} respectively.

Notice that although $\We\in H^1(\Om)$ since $ \We\Big|_1=\We\Big|_0$,
it   does not satisfy the zero Dirichlet
boundary conditions on $\p \Om$. To fix this we introduce the
following boundary-layer corrector to our approximation.

\bl
There exists a corrector $\Ve$ solving the problem
\begin{align}
- \emph{div}(a_\e \nabla \Ve )
&
+ \rho_\e \Ve = 0
\quad \text{in} \quad \Om,
\label{corrector equation}\\
\Ve|_{\p \Om} = -\We|_{\p \Om},
\quad
\Ve\Big|_1
&
= \Ve\Big|_0,
\quad
\frac{\p \Ve}{\p n}\Big|_1 = \e \frac{\p \Ve}{\p n}\Big|_0,
\label{corrector bc}
\end{align}
such that
$\Ue = \We + \Ve \in H^1_0(\Om)$ 
\quad and \quad
$
\D \max_{\bar \Om}| \Ve | \le C \e.
$
\label{corrector estimate}
\el

\bpr
Clearly such solution of
\eqref{corrector equation}, \eqref{corrector bc}
does exist. On each of the  subsets $\Ome_1$
and $\Ome_0$ the coefficients of \eqref{corrector equation}
are smooth. Then the function $\Ve$ can reach its positive maximum or
negative  minimum
only at the boundaries $\tGe$ or $\p \Om$. Let us prove that this
cannot be $\tGe$. Suppose to the contrary the existence of $x_* \in
\tGe$ such that $\D \max_{\bar \Om}  |\Ve|  = |\Ve(x_*)|$. The
strong maximum principle yields that there is no more point inside
$\Ome_1$ or $\Ome_0$ where the maximum is reached. Without  loss
of generality we assume $\Ve(x_*) > \Ve(x)$ for any $x \in \Om
\backslash \tGe$ and $\Ve(x_*) \ge 0$ (otherwise the point would be
a positive maximum for $-\Ve$ and we would then consider $-\Ve$). Then by
the virtue of Hopf's Lemma \cite[p.330]{Evans} applied in the
relevant component of  $\Ome_0$ we have
$$
\left.\frac{\p \Ve}{\p n}\right|_0(x_*) > 0.
$$
From transmission conditions \eqref{corrector bc}
we have that the normal derivative on the $\Ome_1$ side of domain
is also positive. Therefore the value of $\Ve$ increases
from the point $x_*$ inside  $\Ome_1$ in the $n$-direction and hence
$x_*$ is not a point of maximum of $\Ve$ in $\Ome_1$.
The contradiction proves that $| \Ve |$ reaches it's maximum at $\p \Om$.
Then, from boundary conditions \eqref{corrector bc},
$$
\D \max_{\bar \Om}|\Ve| = \max_{\p \Om}|\Ve| =
\max_{\p \Om}|\We| \leq \\
\e \max_{\p \Om}\left| v_1\left(x,\frac{x}{\e}\right)
+ \e v_2\left(x,\frac{x}{\e}\right)
\right|
+
$$
$$
\e \max_{\p \Om}\left| w_1\left(x,\frac{x}{\e}\right)
+ \e w_2\left(x,\frac{x}{\e}\right)
\right|
 \le
C \e.
$$

Obviously $\Ue = \We + \Ve$ satisfies zero boundary condition
on $\p \Om$ and thus belongs to $H^1_0(\Om)$.
\epr

\bl
The constructed corrector $\Ve$ satisfies the estimate
$\| \Ve \|_\Le \le C \e^{3/4}$.
\label{corrector in rho e}
\el

\bpr Let $\chi\in C^\infty(\mathbb{R})$ and $\chi(t)=0,\ t<1$ and
$\chi(t)=1,\ t>2$. Let us define a family of cut-off functions:
$$
\chi_\e(x) = \chi\left( \e^{-1/2} \dist(x,\p \Om) \right),
\quad x \in  \Om.
$$
Then $\chi_\e: \Om \to \Real$ satisfies the properties
\begin{itemize}
\item $\chi_\e(x) = 0$ if $\dist(x,\p \Om) \le \e^{1/2}$,
\item $\chi_\e(x) = 1$ if $\dist(x,\p \Om) \ge 2\e^{1/2}$,
\item $| \nabla \chi_\e| \le C \e^{-1/2}$
and $|\text{supp}\,{\nabla \chi_\e}| \le C \e^{1/2}$,
\end{itemize}
where ``$\text{supp}$'' denotes a function support,
and $|\text{supp}\,{\cdot}|$
is the measure of the corresponding support.
Multiplying \eqref{corrector equation} by $\chi_\e^2 \Ve $ and
integrating by parts, we obtain \be \intl_{\Om} a_\e \nabla \Ve
\cdot \nabla(\chi_\e^2 \Ve ) \, dx + \intl_{\Om} \rho_\e \chi_\e^2
\Ve^2 \, dx = 0. \label{cutoff identity} \ee Then using the identity
$$
\nabla \Ve \cdot \nabla(\chi_\e^2 \Ve )
=
| \nabla(\chi_\e \Ve) |^2 - \Ve^2 | \nabla \chi_\e |^2,
$$
we get from  \eqref{cutoff identity}
 \be \intl_{\Om} a_\e |
\nabla(\chi_\e \Ve) |^2 \, dx + \intl_{\Om} \rho_\e \chi_\e^2 \Ve^2
\, dx = \intl_{\Om} a_\e \Ve^2 | \nabla \chi_\e |^2 \, dx,
\label{cutoff identity 2} \ee
implying
\be \intl_{\Om} \rho_\e \chi_\e^2
\Ve^2 \, dx \le \intl_{\Om} a_\e \Ve^2 | \nabla \chi_\e |^2 \, dx.
\label{cutoff estimate 1} \ee Lemma \ref{corrector estimate}
provides the estimate $\Ve^2 \le C \e^2$. Moreover,
$|\text{supp}\,{\nabla \chi_\e}| \le C \e^{1/2}$ and $| \nabla
\chi_\e |^2 \le C \e^{-1}$. Therefore estimate \eqref{cutoff
estimate 1} yields
    \be
    \| \chi_\e \Ve \|_\Le^2 = \intl_{\Om} \rho_\e
    \chi_\e^2 \Ve^2 \, dx \le C \e^{3/2}.
    \label{cutoff estimate 2}
    \ee
Similarly we estimate
    \be
    \| (1 - \chi_\e) \Ve \|_\Le^2 = \intl_{\Om}
    \rho_\e (1 - \chi_\e)^2 \Ve^2 \, dx \le C \e^{3/2},
    \label{aaa}
    \ee
since $|\text{supp}\,{(1- \chi_\e)}| \le C \e^{1/2}$
and $| \rho_\e | \le C \e^{-1}$.
Combining \eqref{cutoff estimate 2}  and \eqref{aaa}, we obtain $ \|
\Ve \|_\Le = \| (1 - \chi_\e) \Ve + \chi_\e \Ve \|_\Le \le C
\e^{3/4}. $ \epr

\bl
\label{trace estimate}
If $\f \in H^1_0(\Om)$ then
\be
    \Big( \intl_{\tGe}| \f |^2 \,
    dx \Big)^{1/2} \le C  \| \f \|_\He.
    \label{oc}
\ee
\el

\bpr
Extend function $\f$ by zero to whole of
$\mathbb{R}^n$. Then \eqref{oc} follows upon rescaling
$y = x / \e$ from the standard trace estimates applied to each connected
component of $\widetilde{Q}_0$ (which are shifts of $Q_0$).
\epr

\bl
The
corrected approximation $\Ue$ satisfies the estimate
\ben
\| \Be \Ue -
(\Lambda_\e + 1)^{-1} \Ue \|_\Le \leq \| \Be \Ue - (\Lambda_\e +
1)^{-1} \Ue \|_\He \le C \e^{3/4}.
\een
\label{quasimode}
\el

\bpr
For an arbitrary $\f \in \He$ consider
{\small
\begin{align}
|(\Be \Ue
&
- (\Lambda_\e + 1)^{-1} \Ue, \f)_\He |
= |\Lambda_\e + 1|^{-1}
|(\Ue - (\Lambda_\e + 1) \Be \Ue, \f)_\He|
\le
\n \\
& \le
C |(\Ue,\f)_\He - (\Lambda_\e + 1) (\Be \Ue, \f)_\He|
=
C |(\Ue,\f)_\He - (\Lambda_\e + 1) (\Ue, \f)_\Le |
=
\n \\
& =
C \left| \intl_\Om
a_\e \nabla \Ue \nabla \f \, dx
-\Lambda_\e
\intl_\Om
\rho_\e \Ue \f \, dx
\right|
\le
\n \\
& \le C \left| \ \intl_{\Om_0^\e\cup\Om_1^\e} \left( \di ( a_\e
\nabla \Ue ) + \Lambda_\e \rho_\e \Ue \right) \f \, dx \right| + C
\intl_{\tGe} \left| \left. a_\e \frac{\p \Ue}{\p n} \right|_0 -
\left. a_\e \frac{\p \Ue}{\p n} \right|_1 \right| | \f | \, dx.
\label{quasimode estimate 1}
\end{align}
}
Denote the right-hand side of  \eqref{quasimode estimate 1}
by $F_\e(\Ue, \f)$. Substituting $\Ue = \We + \Ve$ and
taking into account \eqref{corrector equation} and
\eqref{corrector bc},
\ben F_\e(\Ue, \f) \le
F_\e(\We, \f) + C \left| (\Lambda_\e + 1) \intl_\Om \rho_\e \Ve \f
\, dx \right| \le F_\e(\We, \f) + C \| \Ve \|_\Le \| \f \|_\Le. \een
By  Lemma \ref{corrector in rho e} and obvious
inequality $\| \f \|_\Le \le \| \f \|_\He$,
 \be F_\e(\Ue, \f)
\le F_\e(\We, \f) + C \e^{3/4} \| \f \|_\He. \label{Fe} \ee
According to  Lemma \ref{estimate in Ome1} $(i)$ and $(ii)$,
\ben F_\e(\We, \f) \le C \e \| \f
\|_{L^2(\Om)} + C \intl_{\tGe} \left| \left. a_\e \frac{\p \We}{\p
n} \right|_0 - \left. a_\e \frac{\p \We}{\p n} \right|_1 \right| |
\f | \, dx.
\een Due to Lemmas \ref{estimate in Ome1} $(iii)$ and \ref{trace estimate} the latter
yields
\be F_\e(\We, \f) \le C \e \| \f \|_\Le + C \e^{3/2} \Big(
\intl_{\tGe}| \f |^2 \, dx \Big)^{1/2} \le C \e \| \f \|_\He.
\label{Fe 4} \ee
Using \eqref{Fe 4}, \eqref{Fe} in
\eqref{quasimode estimate 1} yields
$$
|(\Be \Ue
- (\Lambda_\e + 1)^{-1} \Ue, \f)_\He |
\le C \e^{3/4} \| \f \|_\He
$$
for all $\f \in \He$.
Hence,
$\| \Be \Ue
- (\Lambda_\e + 1)^{-1} \Ue \|_\He \le C \e^{3/4}$.
\epr

\bl
$\| \Ue \|_\Le \ge C
\e^{-1/2}.$
\label{low boundary}
\el

\bpr  By the triangle inequality we  have
    \be \| \Ue \|_\He \ge \| \Ue \|_\Le \ge
\| \We \|_\Le - \| \Ve \|_\Le. \label{low 1} \ee
We consider
{\small
\begin{align}
\| \We \|_\Le^2
& =
\e^{-1}
\intl_{\Ome_0} \left|
w_0\left(x,\frac{x}{\e}\right)
+
\e
w_1\left(x,\frac{x}{\e}\right)
+
\e^2
w_2\left(x,\frac{x}{\e}\right)
\right|^2 \, dx
+
\n \\
& +
\intl_{\Ome_0} \left|
v^0(x)
+
\e
v_1\left(x,\frac{x}{\e}\right)
+
\e^2
v_2\left(x,\frac{x}{\e}\right)
\right|^2 \, dx
=
\n \\
& =
\e^{-1}
\intl_{\Ome_0} \left|
w_0\left(x,\frac{x}{\e}\right)
\right|^2 \, dx
+
O(1)
=
\e^{-1}
\intl_{\Ome_0} \left|
v^0(x) \eta\left(\frac{x}{\e}\right)
\right|^2 \, dx
+
O(1)
, \quad \e \to 0.
\label{We norm}
\end{align}
}
Extending $\eta: Q_0 \to \Real$
by zero onto entire periodicity cell $Q$,
by the mean value property we obtain
\be
\intl_{\Ome_0} \left|
v^0(x) \eta\left(\frac{x}{\e}\right)
\right|^2 \, dx
\quad
\rightarrow
\quad
C_*
:=
\langle \eta^2 \rangle
\intl_{\Om} \left|
v^0(x) \right|^2 \, dx
\quad \text{as} \quad
\e \to 0,
\label{mean value}
\ee
where $\langle \eta^2 \rangle$
is the mean value of function $\eta^2$ over $Q$,
namely
$$
\langle \eta^2 \rangle
=
\intl_Q \eta^2(y) \, dy =
\intl_{Q_0} \eta^2(y) \, dy \,>0.
$$
Since also $v^0(x)\not\equiv 0$,
$C_*$ is positive.
Therefore
\eqref{mean value}
and
\eqref{We norm}
yield
\be
\| \We \|_\Le = C_*^{1/2} \e^{-1/2} + o(\e^{-1/2}),
\quad \e \to 0.
\label{We norm 1}
\ee
Due to Lemma \ref{corrector in rho e}
and \eqref{We norm 1}, it follows
from \eqref{low 1} that
$
\| \Ue \|_\Le
\ge
C_*^{1/2} \e^{-1/2} + o(\e^{-1/2})
$
as
$\e \to 0$.
\epr

\begin{theorem} Let $\lm_0$ be a solution to \eqref{equation for lm0
2} such that $\lm_0\neq\lmd_j$  and $\lm_1$ is defined according to
\eqref{lm1}. Then

 1. For
sufficiently small $\varepsilon>0$
 there exists an eigenvalue $\lme$ of
\eqref{variational} such that
    \be
    |\lme - \lm_0 - \e \lm_1| \le C_1 \e^{5/4}, \label{lme estimation}
    \ee
with constant $C_1$ independent of $\varepsilon$.

2. Let $\We$ be defined by \eqref{Ue} and $\widetilde{\We} = \| \We
\|_\Le^{-1} \We $. Then  there exist constants $c_j(\varepsilon)$
such that
    \begin{equation}
        \Big\| \widetilde{\We} - \sum_{j\in J_\varepsilon}
        c_j(\varepsilon)u_j^\varepsilon \Big\|_{\Le}<C_2\varepsilon^{5/4},
        \label{eigest}
    \end{equation}
where
$J_\varepsilon=\{j:|\lambda_j^\varepsilon-\lm_0 - \e
\lm_1|<C\varepsilon^{5/4}\}$, and
$\lambda_j^\varepsilon$,
$u^\varepsilon_j$
are eigenvalues and ($\Le$-normalized)
eigenfunctions of \eqref{variational}, and the constants $C$ and
$C_2$ are independent of $\varepsilon$.
\label{t1}
\end{theorem}

\begin{proof}
 Application of classical lemma
on ``approximate eigenvalues", e.g. \cite{VishikLyusternik1}, with
$\widetilde{U}_\e = \| \Ue \|_\Le^{-1} \Ue $ as a test function and
$\Lambda_\e = \lm_0 + \e \lm_1$ as an approximate eigenvalue,
ensures,
via Lemmas \ref{quasimode} \& \ref{low boundary},
the existence of  an eigenvalue $\mu_\e$ of operator $\Be$
such that \be | (\Lambda_\e + 1)^{-1} - \mu_\e | \le C \e ^{5/4},
\label{eigenvalue 1} \ee and delivers the estimate analogous to
\eqref{eigest} with $u_j^\varepsilon$ being eigenfunctions of $\Be$
and $\widetilde{\We}$ replaced by
$\widetilde{U}_\e$. It suffices  to notice
that the eigenfunctions of the problem \eqref{variational} and of
operator $\Be$ coincide, their eigenvalues are related via
$\mu_\e^{-1} = \lme +1$ and that $\Le$ norm of the difference between
$\widetilde{U}_\e$ and $\widetilde{\We}$ can be estimated via the right hand
side of \eqref{eigest} (see Lemmas \ref{corrector in rho e} and
\ref{low boundary}).
\end{proof}

\begin{remark}
\label{second remark}
\emph{
Notice that \eqref{eigest} implies weaker but more
transparent interpretations on the approximate eigenfunctions.
For example, introducing
\be
u(x,y) \,=\,\left\{
\ba{ll}
\D v^0\left(x\right), & y\in Q_1,\\[3mm]
\D w_0\left(x,y\right), & y\in Q_0,
\ea
\right.
\label{2sclim}
\ee
we claim that
\be
\Big\| u\left(x,\frac{x}{\e}\right)
-
\sum_{j\in J_\varepsilon}
        d_j(\varepsilon)u_j^\varepsilon
\Big\|_{L^2(\Om)} \le C \e^{3/4},
\label{corrected theorem}
\ee
with appropriate $d_j(\varepsilon)$.
Note that $\| u( \cdot, \frac{\cdot}{\e})\|_{L^2(\Om)} \ge C_0 > 0$.
Then \eqref{corrected theorem} follows from
\eqref{eigest} by splitting its left hand side into the parts
corresponding to $\Ome_1$ and $\Ome_0$,
removing the weight, retaining only the main-order terms and
then adding the inequalities up.
}
\end{remark}

We also remark that, in principle, the result
\eqref{eigest} on the convergence of eigenfunctions
could be further sharpened,
e.g. using the technique of two-scale convergence,
cf. Section \ref{section: convergence} below and \cite{Cherd}.

\subsection{Case (b)}
In this section we assume that $\lm_0=\lmd_j$ for some $j$, its
multiplicity is equal to $1$ and the corresponding eigenfunction $\phi$
has zero mean, i.e. $\langle\phi\rangle=0$,
see Section \ref{section: formal case b}.

\begin{theorem}
\label{t2}
Let $c\in C^3(\Omega)$, $c=0$ on $\p \Omega$,
$\lm_0$
be not  a solution to \eqref{equation for lm0 2}
and $\lm_1$ be defined according to  \eqref{lm1p}.
Then there
exist $\e_0>0$  and  constants $C,C_1$ independent of $\varepsilon$
(but dependent on $c$) such that for any $0<\e\leq \e_0$,

 1.
 There exists an eigenvalue $\lme$ of
\eqref{variational} such that
    \be
    |\lme - \lm_0 - \e \lm_1| \le C \e^{5/4}. \label{lme estimationp}
    \ee

2. Let $\We^*$ be defined by \eqref{w'} and $\widetilde{\We^*} = \|
\We^* \|_\Le^{-1} \We^* $. Then  there exist constants
$c_j(\varepsilon)$ such that
    \begin{equation}
        \Big \|\widetilde{\We^*}-\sum_{j\in J_\varepsilon}
        c_j(\varepsilon)u_j^\varepsilon \Big\|_{\Le}<C_1\varepsilon^{5/4},
        \label{eigestp}
    \end{equation}
where $J_\varepsilon=\{j:|\lambda_j^\varepsilon-\lm_0 - \e
\lm_1|<C\varepsilon^{5/4}\}$, and $\lambda_j^\varepsilon,
u^\varepsilon_j(x)$ are eigenvalues and ($\Le$-normalized)
eigenfunctions of \eqref{variational}.
\end{theorem}

\bpr
Proof of this theorem literally follows the proof of Theorem
\ref{t1} with reference to Lemma \ref{est p}.
\epr

A direct analogue of Remark \ref{second remark}
also holds.

\section{On the eigenfunction convergence}
\label{section: convergence}

In this section we give a brief sketch of further refinement of the presented
results using the technique of two-scale convergence, \cite{Ngu, All, Zhikov2000}.

First, the inclusions intersecting or touching the boundary are ``excluded'',
e.g. by re-defining $a_\e$ and $\rho_\e$ there as in the matrix phase
($a_\e(x)=\rho_\e(x)=1$). Denoting now via $\e\to 0$ an appropriate subsequence in $\e$,
without relabelling, let $u_\e$ and $\lm^\e$ be eigenfunctions and eigenvalues
of the original problem, with normalization
\be
\int_{\Om_1^\e}\nabla u_\e^2\,+\,\e^2\int_{\Om_0^\e}\nabla u_\e^2\,=\,1.
\label{normln}
\ee
The boundedness of $u_\e$ in $L^2(\Om)$ is then implied by \eqref{normln} e.g. via
the uniform positivity of the double-porosity operator whose form is given by the
left hand side of \eqref{normln},
\cite[Thm 8.1]{Zhikov2000}. This implies that, up to a subsequence,
$u_\e\stackrel{2}\rightharpoonup u(x,y)$ and
$\e\nabla u_\e\stackrel{2}\rightharpoonup\nabla_y u(x,y)$, where $u\in L^2(\Omega, H^1_{per})$ and
$\stackrel{2}\rightharpoonup$ denotes weak two-scale convergence. Additionally,
since \eqref{normln} implies $\e\|\nabla u_\e\|_{L^2(\Om_1^\e)}\to 0$,
\cite[Thm 4.1]{Zhikov2000} assures that the two-scale limit is independent of $y$
in the matrix, i.e.
is exactly in the form \eqref{2sclim}.
Further, by \cite[Thm 4.2]{Zhikov2000}, $v^0\in H^1_0(\Omega)$ and
\be
\theta^\e_1\nabla u_\e\,\stackrel{2}\rightharpoonup
\theta_1(y)(\nabla v^0(x)+p(x,y)),
\label{2scgrad}
\ee
where $p\in L^2(\Omega, V_{pot})$ with
$\theta^\e_1$ and $\theta_1(y)$  denoting the characteristic functions of
$\Omega_1^\e$ and $Q_1$, respectively, and
$V_{pot}$ denoting the space of potential vector fields on $Q_1$, i.e.
with respect to the Lebesgue measure supported on $Q_1$, cf.
\cite[\S3.2]{Zhikov2000}.

Let $\lm^\e\to \lm_0$ and $(\lm^\e- \lm_0)/\e\to\lm_1$. Selecting then in \eqref{variational}
appropriate oscillating test functions $\phi=\phi_\e$ one can pass to the limit
recovering the weak forms of the equations derived in Section \ref{fae}.
For example, selecting $\phi_\e(x)=\e \psi(x)b(x/\e)$,
$\psi\in C^\infty_0(\Omega)$, $b(y)\in C^\infty_{per}(Q)$
yields
\[
\int_\Om\int_{Q_1}(\nabla v^0(x)+p(x,y))\cdot\nabla_yb(y)\psi(x)dydx+
\int_\Om\int_{Q_0}\nabla_y w_0(x,y)\cdot\nabla_yb(y)\psi(x)dydx\,=
\]
\be
=
\lambda_0 \int_\Om\int_{Q_0} w_0(x,y) b(y)\psi(x)dydx.
\label{2sclim1}
\ee
This can be seen to be a weak form of \eqref{problem w0} and  \eqref{problem v1}.
Selecting further $\phi_\e(x)= \psi(x)$ can be seen, after some careful
technical analysis, to recover \eqref{homogenized problem eq}, \eqref{homogenized problem bc} and
\eqref{nulam1}.

The above implies that as long as $(v^0)^2+w_0^2 \not \equiv 0$, $\lambda_0$, $\lambda_1$,
$v^0$ and $w_0$ can only be those constructed in Section \ref{fae}.
This does not however rule out the possibility that $v^0$ and $w_0$ are both trivial (equivalently,
the two-scale limit $u(x,y)$ is identically zero). Therefore additional
two-scale compactness type arguments are required, cf. \cite[Lemma 8.2]{Zhikov2000}.
In fact, following literally the argument of Zhikov one observes that the two-scale compactness
of the eigenfunctions does hold, i.e. $u_\e\stackrel{2}\rightarrow u(x,y)$, where
$\stackrel{2}\rightarrow$ denotes strong two-scale convergence,
 in particular there is a convergence
of norms:
\be
\|u_\e - u(x,x/\e)\|_{L^2(\Omega)} \to 0\,\, \mbox{  as }\,\e \to 0.
\label{convnorms}
\ee
 However, this in turn does not rule out the
possibility of $\|u_\e\|\to 0$ with the normalization \eqref{normln}, which requires
a separate analysis.

We announce here a partial result with this effect, postponing detailed discussions for future.

\begin{theorem}
\label{2scalecomp}
Let $\lambda_0$ be not an eigenvalue of the Dirichlet problem in $Q_0$, i.e. $\lambda_0\neq
\lambda^D_j$, $j\geq 1$, see \eqref{auxiliary phi}. Then
\begin{enumerate}

\item[(i)] In the above setting, necessarily, $\beta(\lambda_0)\geq |Q_1|\lambda_0$, i.e.
there are gaps developed for small enough $\e$ in the spectrum, containing in the limit at least
$\{\lambda: \beta(\lambda) < |Q_1|\lambda\}$.

\item[(ii)] If $\beta(\lambda_0) = |Q_1|\lambda_0$, necessarily $u(x,y)\not\equiv 0$.
Consequently, $\lambda_1$ can only be one of those described by \eqref{lm1}.
The eigenfunctions converge strongly, in particular \eqref{convnorms} holds. For fixed
$\lambda_0$ and $\lambda_1$, for small enough $\e$
the multiplicity of the eigenvalues $\lme$
near $\Lambda_\e=\lambda_0+\e\lambda_1$
coincides with the multiplicity of $\nu$ as an eigenvalue of
\eqref{homogenized problem eq}, \eqref{homogenized problem bc}.

\end{enumerate}
\end{theorem}

We remark that the above statement does not provide a full analogue of Hausdorff convergence
of the spectra as in the double porosity case \cite[Thm 8.1]{Zhikov2000}. It does ensure however
the existence of the gaps (on Figure \ref{fig-1}, $(\lambda_j^D, \mu_{j+1})$, $j\geq 1$)
and of the spectrum accumulation near the left ends $\mu_j$, $j\geq 1$, of the ``bands''
$[\mu_j, \lambda_j^D]$.
However it does not clarify whether the ``rests'' of the bands,
$(\mu_j, \lambda_j^D]$ could be accumulation points.
We conjecture that they could.  For a
chosen $\lambda_0=\mu_j$ there exist infinitely many $\lambda_1=\lambda_1^{(n)}$
according to \eqref{lm1}, \eqref{nun}, and $\lambda_1^{(n)}\to+\infty$ as $n\to\infty$.
On any band, for any small enough $\e$ there exists a finite but infinitely increasing number
$N(\e)$ of eigenvalues according to \eqref{lme estimation}. The issue is hence, in a sense, whether
$\e\lambda_1^{(n)}$ may become of order one for large $n$ ($n\sim N(\e)$).
For $\lambda_1^{(n)}\sim \e^{-1}$, according to \eqref{lm1} $\nu\sim \e^{-1}$, and hence, formally,
the solutions $v^0$ of the homogenized equation \eqref{homogenized problem eq}  becomes oscillatory
on the scale $x/\e^{1/2}$. One can attempt deriving asymptotic expansions similarly to those in
Section 3, involving this new scale. A preliminary analysis has shown that those have formal
solutions near every point inside the band. More detailed analysis is beyond the scope of
the present work.

\appendix

\section[Derivation of the limit equation]{Derivation of the limit equation for $v_0$.}
\setcounter{equation}{0}
\renewcommand{\theequation}{\mbox{\thesection.\arabic{equation}}}

Since
$$
- \intl_{\Gamma} \frac{\p v_1}{\p \nx}\, dy = - \intl_{\Gamma} \frac{\p v_1}{\p x_j} n_j \, dy =
\intl_{Q_1} \frac{\p^2 v_1}{\p x_j \p y_j}\, dy,
$$
\eqref{solvability v2} transforms to
$$
( \Delta_x v_0 + \lm_0 v_0 )|Q_1| + \intl_{Q_1} \frac{\p^2 v_1}{\p x_j \p y_j} \, dy =
\intl_{\Gamma} \left( \frac{\p w_1}{\p \ny} + \frac{\p w_0}{\p \nx} \right)\, dy.
$$
Taking into account \eqref{representation v1}
and \eqref{case1w0} this becomes
\bea
 & \D( \Delta_x v_0 + \lm_0 v_0 )|Q_1| + \frac{\p^2 v_0}{\p x_j \p x_k} \intl_{Q_1} \frac{\p N_k}{\p y_j} \, dy =&
 \nonumber \\
 & = \D \frac{\p v_0}{\p x_j} \Big(  - \intl_{Q_1} \frac{\p \cN}{\p y_j} \, dy +  \intl_{\Gamma} \eta n_j \, dy \Big)
  + \intl_{\Gamma} \frac{\p w_1}{\p \ny} \, dy. &
  \label{equation v0}
\eea
Since $\eta(y)=1$ on $\Gamma$,
\be
\intl_{\Gamma} \eta n_j \, dy = \intl_{\Gamma} n_j \, dy =0.
\label{simplification 1}
\ee
We introduce homogenized matrix $\Ah=(\Ah_{jk})_{j,k=1}^n$  by \eqref{A hom}.
According to \eqref{representation w1} we have
\be
\intl_{\Gamma} \frac{\p w_1}{\p \ny} \, dy =
\frac{\p v^0}{\p x_j} \intl_{\Gamma} \frac{\p \cM_j}{\p \ny} \, dy +
v^0 \intl_{\Gamma}  \frac{\p \cP}{\p \ny} \, dy +
\lm_1 v^0 \intl_{\Gamma} \frac{\p \cR}{\p \ny}\, dy.
\label{representation w1/n}
\ee
Substituting \eqref{simplification 1} -- \eqref{representation w1/n} into \eqref{equation v0} yields
\be
-\di \Ah \nabla_x v^0 = \nu(\lm_1) v^0 +
\cK_j \frac{\p v^0}{\p x_j}\quad\mbox{in}\quad \Om,
\label{homogenized equation}
\ee
with
$$
\nu(\lm_1) = \lm_0|Q_1| - \lm_1 \intl_{\Gamma} \frac{\p \cR}{\p \ny}\, dy -
\intl_{\Gamma}  \frac{\p \cP}{\p \ny} \, dy
$$
and
\be
\cK_j = \intl_{Q_1} \frac{\p \cN}{\p y_j} \, dy
- \intl_{\Gamma} \frac{\p \cM_j}{\p \ny} \, dy.
\label{cK}
\ee

\begin{lemma}
$\nu(\lm_1)$ depends on $\lm_1$ with a non-zero linear coefficient.
\end{lemma}

\begin{proof}
We estimate the linear coefficient
$$
\cC := - \intl_{\Gamma} \frac{\p \cR}{\p \ny}\, dy.
$$
Note that
\ben
\cC =
\intl_{\Gamma} \Big( \cR \frac{\p \eta}{\p \ny} - \eta \frac{\p \cR}{\p \ny} \Big)\, dy=
\intl_{Q_0} ( \cR \Delta_y \eta - \eta \Delta_y \cR )\, dy
= \intl_{Q_0} \eta^2 \, dy > 0,
\een
where \eqref{problem cR} and \eqref{probleta} have been used.
Thus
$$
\nu(\lm_1) = \cC \lm_1+ \lm_0|Q_1| -
\intl_{\Gamma}  \frac{\p \cP}{\p \ny} \, dy
$$
with positive constant $\cC$ (depending on the choice of $\lm_0$).
\end{proof}

\begin{corollary}
\label{corollary nu}
(i) If $\lm_0=0$ then $\eta(y)\equiv 1$ and hence $\cC=|Q_0|$.

(ii) According to  \eqref{problem cP} we also have
the representation
\be
\nu(\lm_1) = \cC \lm_1+ \lm_0 \Big( |Q_1| + \intl_{Q_0} \cP \, dy \Big).
\label{nulam}
\ee
\end{corollary}

\begin{lemma}
All  $\cK_j$ defined by \eqref{cK} equal zero.
\end{lemma}

\begin{proof}
First we prove an auxiliary identity, namely
\be
\intl_{\Gamma} \left( \frac{\p \cM_j}{\p \ny} \eta -
\cM_j \frac{\p \eta}{\p \ny}\right) \, dy
= 0.
\label{aux1}
\ee
Notice for this that the left-hand side is
$$
\intl_{Q_0} \left( \Delta \cM_j \eta - \cM_j \Delta \eta \right) \, dy =
- \intl_{Q_0} 2 \frac{\p \eta}{\p y_j} \eta \, dy
= - \intl_{Q_0} \frac{\p \eta^2}{\p y_j} \, dy,
$$
where equations \eqref{problem cM} and \eqref{probleta} have been used.
Since $\eta(y)=1$ on $\Gamma$,
$$
\intl_{Q_0} \frac{\p \eta^2}{\p y_j} \, dy =
\intl_{\Gamma} \eta^2 n_j \, dy =
\intl_{\Gamma} n_j \, dy = 0,
$$
which finishes the proof of \eqref{aux1}.

Then consider
$$
\cK_j = - \intl_{\Gamma} \cN n_j \, dy -
\intl_{\Gamma} \frac{\p \cM_j}{\p \ny} \eta \, dy,
$$
which, according to \eqref{problem Nj} and \eqref{aux1}, yields
\be
\cK_j = \intl_{\Gamma} \cN \frac{\p N_j}{\p \ny}  \, dy
- \intl_{\Gamma} \cM_j \frac{\p \eta}{\p \ny}  \, dy.
\label{aux cK}
\ee
Since $\cN$ and $N_j$ are both harmonic in $Q_0$,
\be
\intl_{\Gamma} \cN \frac{\p N_j}{\p \ny}  \, dy =
\intl_{\Gamma} N_j \frac{\p \cN}{\p \ny}  \, dy.
\label{aux2}
\ee
Using  \eqref{problem cM} and \eqref{problem cN} we obtain
\be
\intl_{\Gamma} \cM_j \frac{\p \eta}{\p \ny}  \, dy =
\intl_{\Gamma} N_j \frac{\p \cN}{\p \ny}  \, dy.
\label{aux3}
\ee
Substitution of \eqref{aux2} and \eqref{aux3} into \eqref{aux cK}
proves the lemma.
\end{proof}

Finally we come to the formulation of homogenized problem for the function $v_0$,
which comes from \eqref{homogenized equation} and boundary condition \eqref{boundary condition},
resulting in \eqref{homogenized problem eq}-\eqref{homogenized problem bc}.

\section{Proof of Lemma \protect{\ref{est p}}}
\setcounter{equation}{0}
\renewcommand{\theequation}{\mbox{\thesection.\arabic{equation}}}

We look for  $v_1,w_1,v_2,w_2$  in the form
    \be
        v_1(x,y)=c(x)\mathcal{V}_1(y),
    \ee
    \be
        w_1(x,y)=c(x)\mathcal{W}_1(y)+\frac{\partial c(x)}{\partial
        x_k} \mathcal{Z}^{(k)}_1(y),
    \ee
    \be
        v_2(x,y)=c(x)\mathcal{V}_2(y)+\frac{\partial c(x)}{\partial
        x_k} \mathcal{P}_k(y),
    \ee
    \be
        w_2(x,y)=c(x)\mathcal{W}_2(y)+\frac{\partial c(x)}{\partial
        x_k} \mathcal{Z}^{(k)}_2(y),
    \ee
where  $c$ is an arbitrary smooth function in $\Omega$ and
 $\mathcal{V}_i,\mathcal{W}_i,\mathcal{Z}^{(k)}_i,\mathcal{P}_k$, $i=1,2\ ,\  \ k=1,..,n$, are functions to be found.

 Applying differential operator
${\Ae}$ given by \eqref{leps} to \eqref{w'} in $\Ome_1$ we obtain
{\small
\begin{align}
\di( & a_\e \nabla \We^* ) + \Lambda_\e \rho_\e \We^* =
\n \\
& = \left\{ \e^{-1} c\Delta_y \mathcal{V}_1 + \e^{0} \left(
c\Delta_y \mathcal{V}_2(y)+\frac{\partial c}{\partial
        x_k}\left\{ \Delta_y\mathcal{P}_k(y) + 2 \frac{\p \mathcal{V}_1}{ \p
y_k}\right\} \right) + \right.
\n \\
& + \left.\left. \e^1 \left( 2 \frac{\p^2 v_2}{\p x_j \p y_j} +
\Delta_x v_1  + \lm_0 v_1 \right) + \e^2 (\Delta_x v_2 + \lm_1 v_1
+\lm_0 v_2) + \e^3 \lm_1 v_2 \right\}\right|_{y=\frac{x}{\e}}.
\label{coef estimate 100}
\end{align}
}
Applying next ${\Ae}$ to \eqref{w'} in $\Ome_0$ we
obtain
{\small
\begin{align}
\di( & a_\e \nabla \We^* ) + \Lambda_\e \rho_\e \We^* =
\n \\
 & =
 \left\{   \e^{0} \left(
 c\left[(\Delta_y+\lm_0) \mathcal{W}_1+ \lm_1 \phi\right]+\frac{\partial c}{\partial
        x_k}[(\Delta_y+\lm_0) \mathcal{Z}^{(k)}_1 + 2 \frac{\p \phi}{\p
        y_k}]
          \right) + \right.
\n \\
& + \e^1 \left( (\Delta_y+\lm_0) w_2 + 2 \frac{\p^2 w_1}{\p x_j \p
y_j} + \Delta_x w_0  + \lm_1 w_1 \right) +
\n \\
& + \left.\left. \e^2 \left( 2 \frac{\p^2 w_2}{\p x_j \p y_j} +
\Delta_x w_1 + \lm_1 w_2 \right) + \e^3 \Delta_x w_2
\right\}\right|_{y=\frac{x}{\e}} \label{coef estimate 200}
\end{align}
}
Evaluating the jumps of conormal derivatives on $\Gamma^\e$,
we obtain
{\small
\begin{align}
\left. a_\e \frac{\p \We^*}{\p n} \right|_0 & - \left. a_\e \frac{\p
\We^*}{\p n} \right|_1 =
\e^0 c\left( \frac{\p \phi}{\p \ny} - \frac{\p \mathcal{V}_1}{\p
\ny} \right)
\Big|_{\substack{x\in \tGe
\\
y \in \tG \hspace{1.5mm}}} +
\n
\\
&
+\e^1 \left( c\left\{\frac{\p \mathcal{W}_1}{\p \ny}  - \frac{\p
\mathcal{V}_2}{\p \ny}\right\} +\frac{\partial c}{\partial
        x_k}\left\{\frac{\p
\mathcal{Z}_1^{(k)}}{\p \ny} + n_k\phi - \frac{\p \mathcal{P}_k}{\p
\ny} - n_k\mathcal{V}_1\right\} \right)
\Big|_{\substack{x\in \tGe \\
y \in \tG \hspace{1.5mm}}}
+
\n \\
& + \e^2 \left( \frac{\p w_2}{\p \ny} + \frac{\p w_1}{\p \nx} -
\frac{\p v_2}{\p \nx} \right)
\Big|_{\substack{x\in \tGe \\
y \in \tG \hspace{1.5mm}}} + \e^3 \frac{\p w_2}{\p \nx}
\Big|_{\substack{x\in \tGe \\
y \in \tG \hspace{1.5mm}}}.\label{coef estimate 300}
\end{align}
}
On the other hand function $\We^*$ is required to be continuous, i.e
we have
\be
    \mathcal{W}_1 = \mathcal{V}_1
    \quad   \textrm{on}\quad  \Gamma,
    \label{cont1}
\ee
      \be
      \mathcal{Z}_1^{(k)} = 0
      \quad   \textrm{on}\quad  \Gamma,
      \label{cont2}
      \ee
      \be
      \mathcal{W}_2 = \mathcal{V}_2
      \quad   \textrm{on}\quad  \Gamma,
      \label{cont3}
      \ee
      \be
      \mathcal{Z}_2^{(k)}=\mathcal{P}_k
      \quad   \textrm{on}\quad  \Gamma.
      \label{cont4}
      \ee
Equating to zero the term of order $\e^{-1}$ in \eqref{coef estimate 100}
and the term of order $\e^0$ in \eqref{coef estimate 300},
we obtain problem for $\mathcal{V}_1$:
    \be   \Delta_y \mathcal{V}_1=0\quad   \textrm{in}\quad
    Q_1,\qquad  \frac{\p \mathcal{V}_1}{\p
\ny}=\frac{\p \phi}{\p \ny}\quad  \textrm{on}\quad  \Gamma.  \label{pr1}
    \ee
A solution to this problem exists since $\langle\phi\rangle=0$ and we can present
it as:
    \be
        \mathcal{V}_1=\widetilde{\mathcal{V}}_1 + \widetilde{A},
    \ee
where $<\widetilde{\mathcal{V}}_1>=0$ and $\widetilde{A}$ is a constant which
will be determined later.

 Equating to zero the term of order $\e^0$ in \eqref{coef estimate
 200},
and using \eqref{cont1}, \eqref{cont2} we obtain problems for $
\mathcal{Z}_1^{(k)}$
    \be
    (\Delta_y+\lm_0) \mathcal{Z}_1^{(k)}=- 2 \frac{\p
    \phi}{\p y_k} \quad
    \textrm{in}\quad Q_0,\qquad \mathcal{Z}_1^{(k)}=0\quad  \textrm{on}\quad  \Gamma,
    \label{pr4}
    \ee
which admits an explicit solution
    \be
    \mathcal{Z}^{(k)}_1(y)=-y_k
    \phi(y), \label{zk}
    \ee
and  for $\mathcal{W}_1$
    \be   (\Delta_y+\lm_0) \mathcal{W}_1=- \lm_1 \phi\quad   \textrm{in}\quad
    Q_0,\qquad \mathcal{W}_1=\mathcal{V}_1\quad  \textrm{on}\quad  \Gamma. \label{pr2}
    \ee
A solution to the latter  exists if and only if
    \be \lm_1 = \intl_{\Gamma} \cV_1 \frac{\p \phi}{\p n_y }\,
    dy = -\int_{Q_1}|\nabla \mathcal{V}_1|^2 dy
    =
    -\int_{Q_1}|\nabla \widetilde{\mathcal{V}}_1|^2 dy,
\label{lm1p}
    \ee
 and we can present it in the following way:
    \be
        \mathcal{W}_1 = \widetilde{\mathcal{W}}_1
        + \widetilde{A} \eta,
    \ee
where $\widetilde{\mathcal{W}}_1$ solves problem \eqref{pr2}
with $\mathcal{V}_1$ replaced by $\widetilde{\mathcal{V}}_1$
(a solution exists for the same reason), and $\eta$ solves
\eqref{probleta} (a solution exists since $\langle\phi\rangle=0$).
Notice that
$$
\lm_0 \langle\eta\rangle = -\int_{\Gamma} \frac{\p \eta }{\p
\ny}dy\neq 0,
$$
otherwise $\lambda_0$ would be a solution of
\eqref{equation for lm0 2} which contradicts to the assumptions of this
section.

Equating to zero the term of order $\e^0$ in \eqref{coef estimate 100}
and the term of order $\e^1$ in \eqref{coef estimate 300},
we obtain problems for $\mathcal{V}_2$ and $\mathcal{P}_k$.
For $\mathcal{V}_2$  we have:
    \be   \Delta_y \mathcal{V}_2=0\quad   \textrm{in}\quad
    Q_1,\qquad \frac{\p \mathcal{V}_2}{\p
    \ny}=\frac{\p \mathcal{W}_1 }{\p \ny}\quad  \textrm{on}\quad  \Gamma.
    \label{pr3}
    \ee
A solution to this problem exists if and only if
    \be
    0 = \intl_{\Gamma} \frac{\p \mathcal{W}_1 }{\p \ny}\,
    dy = \intl_{\Gamma} \frac{\p \widetilde{\mathcal{W}}_1 }{\p \ny}\,
    dy + \widetilde{A} \intl_{\Gamma} \frac{\p \eta }{\p \ny}\,
    dy,
    \ee
and consequently
    \be
    \widetilde{A} =
    (\lm_0 \langle \eta \rangle)^{-1}
    \intl_{\Gamma} \frac{\p \widetilde{\mathcal{W}}_1 }{\p \ny}\,
    dy.
    \ee
The problem for $\mathcal{P}_k$ has the form:
    \be
    \Delta_y\mathcal{P}_k(y) =- 2 \frac{\p \mathcal{V}_1}{ \p
    y_k}\quad   \textrm{in}\quad
    Q_1,\qquad  \frac{\p \mathcal{P}_k}{\p \ny}=\frac{\p
    \mathcal{Z}_1^{(k)}}{\p \ny}    - n_k\mathcal{V}_1\quad  \textrm{on}\quad \Gamma.
    \label{pr5}
    \ee
Solvability condition for this problem has the form
    \be
    2\intl_{ Q_1}   \frac{\p \mathcal{V}_1}{ \p
    y_k}\,
    dy=\intl_{\Gamma}\left( \frac{\p
    \mathcal{Z}_1^{(k)}}{\p \ny}   - n_k\mathcal{V}_1\right)\,
    dy.
    \label{conpr5}
    \ee
The left hand side of \eqref{conpr5} can be transformed as follows,
    \be
    \int_{ Q_1}  2 \frac{\p \mathcal{V}_1}{ \p
    y_k}\,
    dy=-2\intl_{\Gamma}n_k\mathcal{V}_1\,
    dy.
    \label{conpr51}
    \ee
On the other hand, for the right hand side of \eqref{conpr5},
    $$
    \int_{\Gamma}\left( \frac{\p
    \mathcal{Z}_1^{(k)}}{\p \ny}   - n_k\mathcal{V}_1\right)\,
    dy=\intl_{\Gamma}\left( -y_k\frac{\p
    \phi}{\p \ny}   - n_k\mathcal{V}_1\right)\,
    dy=
    $$
    \be
    =
    \intl_{\Gamma}\left( -y_k\frac{\p
    \mathcal{V}_1}{\p \ny}   - n_k\mathcal{V}_1\right)\,
    dy=-2
    \intl_{\Gamma} n_k\mathcal{V}_1\,
    dy. \label{conpr52}
    \ee
 Here we used \eqref{zk}, \eqref{pr5} and the integration by parts.
 Comparing \eqref{conpr51} and \eqref{conpr52} we see that solvability
 condition \eqref{conpr5} is satisfied. Finally $\mathcal{W}_2$ and
 $\mathcal{Z}_2^{(k)}$ are arbitrary smooth functions satisfying
 \eqref{cont3} and \eqref{cont4}.  $\,\,\,\,\,\,\,\Box$

\section*{Acknowledgments}
The work was supported by Bath Institute for Complex Systems
(EPSRC grant GR/S86525/01),
 by Nuffield grant NAL/32758, EPSRC grant EP/E037607/1
 and RFBR grant 07-01-00548.
The authors acknowledge partial support of Isaac Newton Institute for
Mathematical Sciences under the programme ``Highly Oscillatory Problems:
 Computation, Theory and Application'' February-July 2007.

\bibliographystyle{plain}

\begin{thebibliography}{25}

\bibitem{All}
\newblock G. Allaire,
\newblock \emph{Homogenization and two-scale convergence},
\newblock SIAM J. Math. Anal., \textbf{23} (1992), 1482--1518.

\bibitem{ADH}
\newblock T. Arbogast, J.Jr. Douglas and U. Hornung,
\newblock \emph{Derivation of the
double porosity model of single phase flow via homogenization theory},
\newblock SIAM J. Math. Anal., \textbf{21} (1990), no. 4, 823--836.


\bibitem{NatGol}
\newblock N. Babych, Yu. Golovaty,
\newblock \emph{Quantized asymptotics of high frequency
oscillations in high contrast media},
\newblock in ``Proc. of Waves 2007,
the 8th Int. Conf. on Mathematical and Numerical
Aspects of Waves (23rd-27th July 2007, Reading)'',
University of Reading - INRIA, (2007), p. 35-37.


\bibitem{BMP}
\newblock A. Bourgeat, A. Mikeli\'c and A. Piatnitski,
\newblock  \emph{On the double porosity
model of a single phase flow in random media},
\newblock Asymptot. Anal., \textbf{34} (2003),
no. 3-4, 311--332.

\bibitem{Cherd}
\newblock M.I. Cherdantsev,
\newblock \emph{
Uniform exponential decay and spectral convergence for high contrast media
with a defect via homogenization}, submitted.

\bibitem{CSZ}
\newblock K.D. Cherednichenko, V.P. Smyshlyaev and V.V. Zhikov,
\newblock  \emph{Non-local homogenised limits
for composite media with highly anisotropic periodic fibres},
\newblock Proc. Roy. Soc. Edinb. A, \textbf{136} (2006), no. 1, 87--114.

\bibitem{Cheredn2}
\newblock K.D. Cherednichenko,
\newblock  \emph{Two-scale asymptotics for non-local effects in composites
with highly anisotropic fibres},
\newblock  Asymptot. Anal., \textbf{49} (2006), no. 1-2, 39--59.

\bibitem{Evans}
\newblock L.C. Evans,
\newblock ``Partial Differential Equations. Graduate Studies in Mathematics, 19",
\newblock American Mathematical Society, Providence, RI, 1998.

\bibitem{FeKh}
\newblock V.N. Fenchenko and E.Ya. Khruslov,
\newblock  \emph{Asymptotic
behaviour or the solutions of differential equations with strongly
oscillating and degenerating coefficient matrix},
\newblock Dokl. Akad. Nauk Ukrain. SSR Ser. A, \textbf{4} (1980), 26--30.

\bibitem{hempel}
\newblock R. Hempel  and K. Lienau,
\newblock \emph{Spectral properties of periodic media in
the large coupling limit},
\newblock Comm. Partial Differential Equations, \textbf{25}
(2000), no. 7-8, 1445--1470.

\bibitem{JKO}
\newblock V.V. Jikov, S.M. Kozlov and O.A. Oleinik,
\newblock ``Homogenization of Differential Operators and Integral
Functionals",
\newblock Springer-Verlag, Berlin, 1994.

\bibitem{KS}
\newblock I.V. Kamotski, V.P. Smyshlyaev,
\newblock \emph{Localised modes due to defects in high
contrast periodic media via homogenization},
\newblock BICS
preprint 3/06, Submitted. Available online at: www.bath.ac.uk/math-sci/preprints/BICS06 3.pdf.

\bibitem{Kh}
\newblock E.Ya. Khruslov,
\newblock \emph{An averaged model of a strongly
inhomogeneous medium with memory},
\newblock Uspekhi Mat. Nauk, \textbf{45} (1990),
no.1, 197--198. English translation in
Russian Math. Surveys, \textbf{45} (1990), no. 1, 211--212.

\bibitem{Kh2}
\newblock E.Ya. Khruslov,
\newblock \emph{Homogenized models of composite media}
\newblock in ``Composite media and
homogenization theory (Trieste, 1990)",
Progr. Nonlinear Differential
Equations Appl., 5, Birkh\"auser Boston, Boston, MA (1991), 159--182.

\bibitem{Ngu}
\newblock G. Nguetseng,
\newblock \emph{A general convergence result for a functional
related to the theory of homogenization},
\newblock SIAM J. Math. Anal., \textbf{20} (1989), 608--623.

\bibitem{Panas}
\newblock G.P. Panasenko,
\newblock \emph{Multicomponent homogenization of
processes in strongly nonhomogeneous structures},
\newblock Math. USSR Sbornik, \textbf{69} (1991), no. 1, 143--153.

\bibitem{Sandr}
\newblock G.V. Sandrakov,
\newblock \emph{Averaging of nonstationary problems in the theory
of strongly nonhomogeneous elastic media},
\newblock Dokl. Akad. Nauk, \textbf{358} (1998), no. 3, 308--311.

\bibitem{VishikLyusternik1}
\newblock M.I. Vishik, L.A. Lyusternik,
\newblock \emph{Regular degeneration and
boundary layer for linear differential equations with a small
parameter},
\newblock Usp. Mat. Nauk, \textbf{12} (1957), no. 5(77), 3--122.

\bibitem{Zhikov2000}
\newblock V.V. Zhikov,
\newblock \emph{On an extension and an application of the two-scale convergence method}, (Russian)
\newblock Mat. Sb., \textbf{191} (2000), no. 7, 31--72;
English translation in
Sb. Math., \textbf{191} (2000), no. 7-8, 973--1014.

\bibitem{Zh2}
\newblock V.V. Zhikov,
\newblock \emph{Gaps in the spectrum of some elliptic operators in divergent form with periodic coefficients},
(Russian)
\newblock Algebra i Analiz, \textbf{16} (2004), no. 5, 34--58;
English translation in
St. Petersburg Mathematical J., \textbf{16} (2005), no. 5, 773--790.

\end{thebibliography}


\vskip1cm

{\scriptsize

N.Babych@bath.ac.uk

I.V.Kamotski@bath.ac.uk

V.P.Smyshlyaev@maths.bath.ac.uk

}
\medskip

Department of Mathematical Sciences,
 University of Bath, Bath BA2 7AY, UK


\end{document}